\documentclass[serif]{wiley-article}
\usepackage{url, bm, bbm}

\usepackage{hyperref}
\usepackage{natbib}

\numberwithin{equation}{section}

\def\mc{\mathcal}
\def\kn{n/2}
\def\mb{\mathbb}

\def\ba{\,|\,}

\newcommand{\N}{\mb N}
\newcommand{\R}{\mb R}

\newcommand{\Z}{\mb Z}

\newcommand{\p}{\mb P}

\newcommand{\cC}{\mc C}

\newcommand{\cG}{\mc G}

\newcommand{\cK}{\mc K}
\newcommand{\cL}{\mc L}

\newcommand{\cN}{\mc N}

\newcommand{\cP}{\mc P}
\newcommand{\cR}{\mc R}

\newcommand{\bet}{\begin{theorem}}
\newcommand{\ent}{\end{theorem}}
\newcommand{\bepr}{\begin{proposition}}
\newcommand{\enpr}{\end{proposition}}
\newcommand{\bel}{\begin{lemma}}
\newcommand{\enl}{\end{lemma}}
\newcommand{\bep}{\begin{proof}}
\newcommand{\enp}{\end{proof}}

\newcommand{\E}[1]{\mathbb E\left [ #1 \right ]}

\renewcommand{\phi}{\varphi}
\renewcommand{\le}{\leqslant}
\renewcommand{\ge}{\geqslant}

\newcommand{\intd}[1]{\mathrm{d}#1}

\newcommand{\1}[1]{\, 1\! \left\{ #1 \right\} }

\newcommand{\diam}{\operatorname{diam}}
\newcommand{\wt}{\widetilde}

\allowdisplaybreaks 
\def\vp{\varphi}
\def\e{\varepsilon}
\def\D{D}
\def\Var{\ms{Var}}
\def\Cov{\ms{Cov}}
\def\s{\sigma}

\def\ms{\mathsf}
\def\a{\alpha}
\def\e{\varepsilon}
\def\ff{\infty}

\def\b{\beta}
\def\th{\theta}
\def\bt{\b_{\to, n}}
\def\Jt{J_\to}
\def\Eb{E_{\ms b}}
\def\Ed{E_{\ms d}}

\def\es{\varnothing}
\def\d{{\mathrm d}}

\def\bnz{\b_{i, n}}
\def\dnz{\De_{i, n}}

\def\btr{\b_{\to, n}^{r, s}}
\def\btr{\b_{\to, n}^{r, s}}
\def\bn{\b_n}
\def\de{\delta}
\def\De{\Delta}
\def\EE{\mathbb E}
\def\P{\mathbb P}

\def\one{\mathbbmss1}
\def\ul{\underline}

\papertype{}

\title{
	Functional central limit theorems for persistent Betti numbers on cylindrical networks 
	}

	\author[1]{Johannes Krebs}
	\affil[1]{Institute of Applied Mathematics, Heidelberg University}

	\author[2]{Christian Hirsch}
	\affil[2]{Bernoulli Institute, University of Groningen}

	\corraddress{Heidelberg University, Im Neuenheimer Feld 205, 69120 Heidelberg, Germany.\\Bernoulli Institute, University of Groningen, Nijenborgh 9, 9747 AG Groningen, The Netherlands.}
	\corremail{krebs@uni-heidelberg.de\\ \hspace{3em}c.p.hirsch@rug.nl}
	\fundinginfo{This research was partially supported by the Deutsche Forschungsgemeinschaft (DFG), grant number KR 4977/1-1.}

	\runningauthor{Krebs and Hirsch}

\begin{document}

\maketitle
\begin{abstract}
\setlength{\baselineskip}{1.8em}
We study functional central limit theorems for persistent Betti numbers obtained from networks defined on a Poisson point process. The limit is formed in large volumes of cylindrical shape stretching only in one dimension. The results cover a directed sublevel-filtration for stabilizing networks and the \v Cech and Vietoris-Rips complex on the random geometric graph. 

The presented functional central limit theorems open the door to a variety of statistical applications in topological data analysis and we consider goodness-of-fit tests in a simulation study.
\end{abstract}

\section{Introduction}

\label{int_sec}
%
%
Topological data analysis (TDA) relies on an equally simple as appealing principle: Leverage invariants from algebraic topology to extract surprising insights from data. Although, a priori it is not at all apparent that this unconventional idea delivers added value, it has now been adapted in an impressively diverse range of domains, such as astronomy, materials science, biology and finance (\cite{rien, hiraoka,gidea2018topological}). Often TDA-based methods unearth relations of an entirely different nature than those found with more conventional methods.

%
%
Although TDA has the power to produce compelling visuals, it is often not at all clear, whether the purported effects genuinely come from pivotal characteristics of the data set, or whether they are a mere incarnation of chance. In other words, despite its widespread dissemination across disciplines, for large parts of topological data analysis, the development of statistically sound testing procedures is still in its infancy.

%
%
One route towards devising goodness-of-fit tests is to rely on Monte Carlo methods (\cite{biscioTDA, turner}). In the most immediate approach, this would mean fixing a sampling window and comparing summary statistics of a given data set to those from a large number of samples under the null model. However, for large sampling windows, Monte Carlo methods require massive computational efforts, which are additionally tied to the specific sampling window. Hence, statistical tests that become asymptotically precise in large domains would be the ideal complement to the Monte Carlo tests. Although asymptotically precise tests decouple the computation of the test statistic under the null hypothesis from the size of the sampling window, a possible practical problem is that a robust estimation of the limiting variances could still entail simulations on relative large windows. An alternative could be to start from the precise integral expressions for the limiting variances and covariances derived by \cite{svane} and devise reasonable strategies to approximate them.

%
%
A challenge in establishing functional CLTs is that TDA captures highly subtle characteristics of the underlying data, so that refined topological arguments are needed, before the problem becomes amenable to established limit-frameworks from stochastic geometry such as \cite{baryshnikov2005gaussian, penrose2001central}. Following this path made it possible to derive large-volume central limit theorems for (persistent) Betti numbers (\cite{krebs2019asymptotic, trinh2018central}). In a parallel line of research, asymptotics for the expected persistent Betti numbers could be established in the subcritical regime by \cite{BauerBetti}.

%
%
While previous works achieved pivotal progress towards analyzing point patterns through statistically well-established methodology, in a variety of applications the data is not given by mere points, but rather in the form of richer geometric objects. For instance, \cite{marron} analyze the directed network of brain arteries by working with the sub-level filtration in a distinguished direction. These questions motivate us to establish functional central limit theorems for TDA-based methods on network data that provide the foundation for rigorous statistical testing. Although asymptotic normality was shown previously, the new functional CLTs are a great advantage. Indeed, they provide practitioners with a large degree of freedom to construct goodness-of-fit tests tailored to their needs. By the continuous mapping theorem, the only requirement is that the summary statistic varies continuously in the persistence diagram. For instance, \cite{svane} use the accumulated persistence function to derive a TDA-based goodness-of-fit tests, which is then applied to a dataset from neuroscience.

%
%
The class of possible random networks is of course enormous and accommodates a variety of entirely different structures. When thinking in the context of tree-shaped networks, highly relevant information is encoded in the merging pattern of different branches. That is, for trees growing into a preferred direction, TDA allows to track at which levels new branches appear and how long they survive before merging into an already existing branch. Stochastic geometry offers an ample variety of connection rules leading to tree-based networks. Among the most prominent examples are the minimal spanning tree, the Poisson tree and the directed spanning tree (\cite{mst, tree1, radSpan}). In fact, as we will see in Theorem \ref{T:GaussianLimitPoisson1}, the method for establishing the limit result does not hinge on the tree structure and remains true for more general networks.

 \begin{figure}[!htpb]
 \centering
 \includegraphics[trim={3cm 0cm 3cm 0cm}, clip, width=0.69\textwidth]{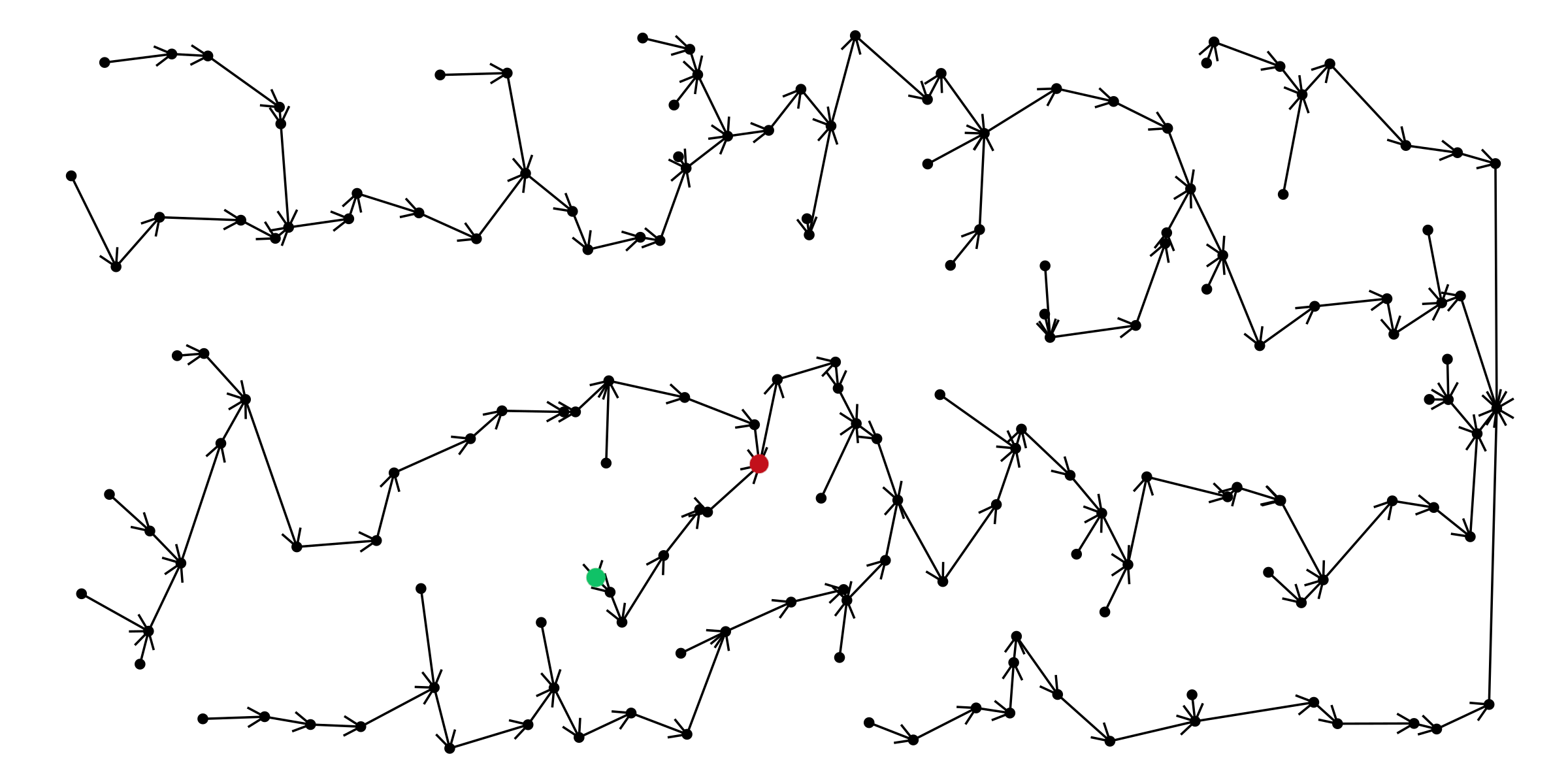}
 \caption{Directed spanning forest on a Poisson point process. At the green point a new component is born that dies at the red point.}
 \label{dsf_fig}
\end{figure}  

%
%
Apart from these directed sublevel filtrations, we can also interpret the classical \v Cech (\v C) and Vietoris-Rips (VR) filtration of a point pattern within the framework of network-based TDA. Indeed, connecting any two points that are closer than a given distance $r$ leads to the {Gilbert graph} (\cite{gilbert, penrose}), one of the most prototypical examples of spatial random networks. {The \v C- and VR-complex can then be considered as extensions of the Gilbert graph, which aim at detecting higher-dimensional correlations in the random geometric graph. Here the \v C-complex features additionally the homotopy equivalence to the union of balls with radius $r$, centered at the points of the underlying point pattern.} In this setting, it is not just the merging pattern of components that is of interest. In fact, the trademark of TDA are the life times of loops, holes and higher-dimensional features. Despite the geometric richness of these features, they can all be thought of as features derived from the Gilbert graph through means of filtrations.

%
%
A pivotal assumption in our study is that the networks are confined to a cylindrical shape growing to infinity only in one of the space directions. This restriction is not simply a matter of convenience, but lies at the core of the proof of the functional CLT. Indeed, looking for instance at the { \v C- and }VR complex, we witness the emergence of long-range dependencies when passing through the critical radius of continuum percolation. A similar effect was observed by \cite{svane}; there, it was resolved by considering only bounded features.
 
 The proof of the functional CLT consists of two steps: showing multivariate normality and tightness. The proof of the multivariate normality relies on the stabilization framework of \cite{penrose2001central}. For the tightness, we resort to the Chentsov-type tightness criterion of \cite{bickel1971convergence} involving fourth-moments of block-increments of the persistent Betti numbers. Due to the intricate geometries, a particular challenge arises when verifying this condition for very small blocks. To that end, we rely on the technique of \cite{davydov}, { which is also detailed in \cite{bickel1971convergence}, }allowing to reduce the verification of the condition to blocks arranged in a grid of suitable spacing.


The rest of the manuscript is organized as follows. First, Section \ref{mod_sec} introduces the precise conditions on the considered models and the functional CLTs. Next, Section \ref{ex_sec} elucidates that a variety of models from literature are covered by this framework. In Section \ref{sim_sec}, we present a simulation study illustrating that the asymptotic limiting results are already accurate in moderately-large sampling windows. The proof of the functional CLTs are outlined in Section \ref{proof_sec}, where we focus on tightness, which is the major methodological contribution of the present paper. The proof resides on two essential steps: 1) reduction to a grid, 2) moment bounds using suitable covariance bounds. Finally, in an appendix, we expound on how to adapt to the present setting previous arguments for multivariate asymptotic normality from literature.

%
\section{Model and main result}
\label{mod_sec}

%
%
As outlined in Section \ref{int_sec}, the main results in this paper are about networks embedded in a cylindrical space $\R \times A $ for some non-empty compact convex set $A \subset \R^{d - 1}$. Furthermore, the networks are all based on suitable restrictions of a homogeneous Poisson point process $\cP = \{Z_i\}_{i \ge 1}$ in $\R \times A$. By proper rescaling, we may assume $\cP$ to have unit intensity. In the following, we consider networks constructed on the restriction $\cP_n$ of $\cP$ to the sampling window $W_n = [-n/2, n/2] \times A$. 

%
%
We start by presenting the functional CLT for the directed filtration on networks. This quantity has the advantage of involving only the appearance and merging of branches, so that we can state the result without needing to invoke the elaborate machinery behind general persistent homology. Since we are considering a distinguished direction, it is natural to say that an edge $\{Z_1, Z_2\}$ is \emph{outgoing} from a vertex $Z_1$ if $Z_2$ lies to the right of $Z_1$.

To make this precise, we assume that the considered network emerges from the underlying point process by a construction rule that is covariant in the $x$-direction: if $\mc E(Z_i, \cP_n)$ denotes the family of outgoing edges attached to the node $Z_i \in \cP_n$, then 
$$\mc E(Z_i + x, \cP_n + x) = \mc E(Z_i, \cP_n)$$
for every $x \in \R$.

If $Z_i, Z_j \in \cP_n$ are connected by an edge in the network and $Z_j$ lies to the left of $Z_i$, then we call $Z_j$ a \emph{parent} of $Z_i$. With this terminology, a point $Z_i \in \cP_n$ gives birth to a new component if $Z_i$ does not have parents. Writing $\pi_1: \R \times A \to \R$ for the projection to the $x$-coordinate, $\pi_1(Z_i)$ is the \emph{birth time} of that component. On the other hand, if $Z_i$ has parents in several different components, then only one of them survives, namely the one with the smallest birth time. Then, $\pi_1(Z_i)$ is the \emph{death time} of the other components. We refer to Figure~\ref{dsf_fig} for an example of points giving birth to a new component, or joining two or more components, respectively.

%
%
This gives rise to a collection $\{(B_i, D_i)\}_{i \ge 1}$ of birth-death pairs and we write
$$\btr= \#\{i:\, (B_i, D_i - B_i) \in [-n/2, rn] \times [s, \ff)\}.$$
for the number of components that are born before time $rn$ and live for a time at least $s$.

Moreover, we assume the construction rule of the network to be stabilizing in the vein of \cite{baryshnikov2005gaussian,lee1997central}. Loosely speaking, the decision whether or not to put an edge between two points depends only on a bounded, possibly random neighborhood. More precisely, we assume that there exists an almost surely finite \emph{radius of stabilization} $R > 0$ such that 
\begin{align}
\label{stabEq1}
  \mc E\big(Z_i, ( \cP \cap( W_R \setminus W_1))\cup \mc A \cup \mc B\big) = \mc E\big(Z_i, (\cP \setminus W_1) \cup \mc B)
\end{align}
holds for all finite $\mc A\subset (\R \setminus W_R) \times A$, $\mc B \subset W_1 \times A$ and $Z_i \in \mc B$. We assume that $R$ has exponentially decaying tails in the sense that there exist $c > 0$ such that
\begin{align}
\label{stabEq2}
	\P(R > r) \le e^{-cr}
\end{align}
holds for all $r > 1$.
Henceforth, we track the life time of a feature until a  fixed life time $T$, so that the rescaled persistence diagram lives on the space $\Jt = [-1/2, 1/2] \times [0, T]$.
%
%
\begin{theorem}[Functional CLT for directed filtrations]
\label{T:GaussianLimitPoisson1}
Let $T > 0$ and consider a network that is exponentially stabilizing in the sense of \eqref{stabEq1} and \eqref{stabEq2}. Then, the process
\[
		\Bigg\{\frac{ \btr- \E{\btr}}{\sqrt n}\Bigg\}_{(r, s) \in \Jt}
\]
converges in the Skorokhod topology to a mean-zero Gaussian process.
\end{theorem}

%
%
Next, we move to the \v C- and VR complex. As announced in Section \ref{int_sec}, the latter filtrations are intimately connected to the Gilbert graph. To that end, we first describe in detail the persistent Betti numbers, a key characteristic of TDA. At the foundation of persistent Betti numbers are the standard Betti numbers encoding information about loops, holes and higher-dimensional topological features of the underlying space. 

%
%
Mathematically, the concept of features of different dimension is captured through the machinery of \emph{simplicial complexes} as described by \cite{harer}. Loosely speaking, a simplicial complex $\cK = \cK(\cP)$ constructed from $\cP$ is an abstract combinatorial structure consisting of points, edges, triangles and the corresponding higher-dimensional simplices. The goal is to define the adjacency structure in such a way that it resembles closely the topology of an object of interest. 

{Two of the most prominent examples are the \emph{\v C} and the \emph{VR complex}. Here, for $r > 0$ and $n \ge 1$, the first complex is given by
 \begin{align*}
				\cC_r(\cP_n) &:= \{ \s \subset \cP_n: \cap_{x\in\s} B(x,r) \ne \es \}, 
 \end{align*}
 where $B(x,r)= \{y\in\R^d: \|x - y\|\le r\}$ is the closed $d$-dimensional ball with radius $r$ and center $x$. The second complex is} defined as
 \begin{align*}
				\cR_r(\cP_n) &:= \{ \s \subset \cP_n: \diam(\s)\le 2r \}, 
 \end{align*}
where $\diam$ is the diameter of a measurable set. 

Building a free $\Z/2$-vector space on all simplices of a simplicial complex $\cK$, the key towards computing persistent Betti numbers is a specific map $\partial$ relating different dimensions, see \cite{harer}. The kernel $Z_q(\cK)$ and the image $B_q(\cK)$ of this map in dimension $q \le d$ are also known as \emph{cycle group} and \emph{boundary group}, respectively. Then, the $\Z/2$-vector-space dimension 
$$\b_q := \ms{dim}(Z_q(\cK) / B_q(\cK))$$
defines the \emph{$q$th Betti number}. 

%
%
The examples of the {\v C- and the} VR-filtration point already the route towards the concept of persistent Betti numbers. Here, we do not only find a single simplicial complex, but an entire filtration parameterized by the radius $r$. This opens the door towards tracking the time points when certain features appear and when they disappear again.

More precisely, a \emph{filtration} is an increasing family of simplicial complexes $\cK = (\cK_r: r \ge 0) $. The corresponding $q$th persistent Betti number for parameters $(r, s) \in \De = \{ (r, s): 0\le r\le s\}$ is defined by
\begin{align}\label{Def:PersistentBetti0}
	\b^{r, s}_{\cK, q,n} \coloneq\dim \big(Z_q (\cK_r)\big) - \dim \big( B_q (\cK_s) \cap Z_q (\cK_r)\big).
\end{align}
Henceforth, we consider features of bounded birth- and death times. 
More precisely, we will deal with the following empirical processes: On the rectangle $J_{\cK} := [0, T]^2$, consider the $q$th persistence diagram 
$
		\ms{PD}_q := \sum_{i\ge 1} \delta{(B_i,D_i)},
$
which relies on the persistent Betti numbers from \eqref{Def:PersistentBetti0} given the \v C- or  VR-filtration built on the point cloud $\cP_n$,  in the sense that 
 $$    \b^{r, s}_{\cK, q,n} = \ms{PD}_q\big([0, r] \times [s, \infty)\big).$$
%
%

\begin{theorem}[Functional CLT for the \v C- and the VR-filtration]
\label{T:GaussianLimitPoisson2}
Consider the \v C- or VR-complex. Let $1 \le q \le d - 1$. Then,
\[
	\Bigg\{\frac{\b^{r, s}_{\cK, q,n} - \EE[\b^{r, s}_{\cK, q,n}]}{\sqrt n}\Bigg\}_{(r, s) \in J_\cK}
\]
converge in the Skorokhod topology to  mean-zero Gaussian processes.
\end{theorem}

A slight nuisance of Theorem \ref{T:GaussianLimitPoisson2} is the special case $q = 0$, corresponding to the $0$-dimensional features. In contrast to the setting with $q \ge 1$, all features are born at time 0, thereby resulting in a one-dimensional instead of a two-dimensional process. In order to avoid cumbersome case distinction {in the proof of this result}, we discuss in detail the setting where $q \ge 1$, noting that the general techniques seamlessly extend to $q = 0$. 

Many steps of the proof of Theorems \ref{T:GaussianLimitPoisson1} and \ref{T:GaussianLimitPoisson2} are almost identical, since also in the \v C- and VR-filtration, it is possible to define the notion of a radius of stabilization. Indeed, let $m$ be the smallest positive integer such that all connected components of the filtration at level $T$ intersecting $W_1$ lie inside $W_m$ and set $R \coloneq m+2T$. Then, changing the Poisson process $\cP$ outside $W_R$ does not change features involving points in $W_1$. Moreover, if we let $S$ denote the smallest positive value such that $([S, S + T] \times A) \cap \cP = \es$, then the void probabilities of the Poisson point process imply that there exists $c > 0$ such that $\P(S > r) \le e^{-cr}$ holds for all sufficiently large $r$. Hence, the radius of stabilization $R$ also exhibits exponential tails in the present setting.

%
%
\section{Examples for stabilizing networks}
\label{ex_sec}
In this section, we illustrate that the central stabilization assumption in Theorem \ref{T:GaussianLimitPoisson1} holds for a variety of spatial random networks. 

%
%
The most basic example is the Gilbert graph, where two nodes of $\cP$ are connected if they are closer than a given distance $r > 0$. In particular, we may take $2\lceil r \rceil + 2$ as finite radius of stabilization. The same radius of stabilization can be used if we additionally delete edges according to a distance-dependent probability.

Next, we define the \emph{directed spanning forest} $G(\vp)$ on a vertex set $\vp \subset \R \times A$ by drawing a single edge from $x \in \vp$ to the closest Euclidean neighbor $y \in \vp$ to the right of $x$ (\cite{radSpan, coupier}). That is, 
$$y = \ms{argmin}_{\substack{y' \in \vp \\ \pi_1(y') \ge \pi_1(x)}}|x - y'|.$$
If $\varphi$ has a right-most point, then we formally define its outgoing edge to end at the same point. To define the radius of stabilization, we first let 
$R^* := \inf\{\pi_1(Z_i):\, \pi_1(Z_i) \ge 1\}$
denote the $x$-coordinate of the first point of $\cP$ lying to the right of 1. Then, 
$R := \ms{diam}(A) + R^*$
is a radius of stabilization since every point lying to the right of $R$ cannot be the right-closest Euclidean neighbor of a point in $W_1$. Furthermore, $R^*$ and thereby $R$ are exponentially stabilizing since for $r > 1$
$$\P(R^* > r) = \P(\cP \cap ([1, r]\times A) { = \emptyset}) = \exp(-(r - 1)|A|).$$

%
%
\section{Simulation study}

\label{sim_sec}
In this section, we illustrate the functional CLTs through simulations. First, Section \ref{dsf_sec} showcases at the hand of the directed spanning tree presented in Section \ref{ex_sec} that the asymptotic normality from Theorem \ref{T:GaussianLimitPoisson1} is already accurate in moderately large sampling windows. Second, in Section \ref{cech_sec}, we proceed to the \v Cech-filtration.
For the simulations, we rely on a homogeneous Poisson point process with intensity 2 in a $15 \times 5$ sampling window.

  %
  %
   \subsection{Directed spanning tree}
   \label{dsf_sec}
   First, we illustrate the asymptotic normality from Theorem \ref{T:GaussianLimitPoisson1} at hand of the directed spanning forest. To begin with, Figure \ref{pd_net_fig} highlights the persistence diagrams for one realization of the network model based on a Poisson point process. In particular, only few components live for an exceptionally long time.
      \begin{figure}[!htpb]
 \centering
 \includegraphics[trim={0cm 17cm 40cm 0cm}, clip, width=0.59\textwidth]{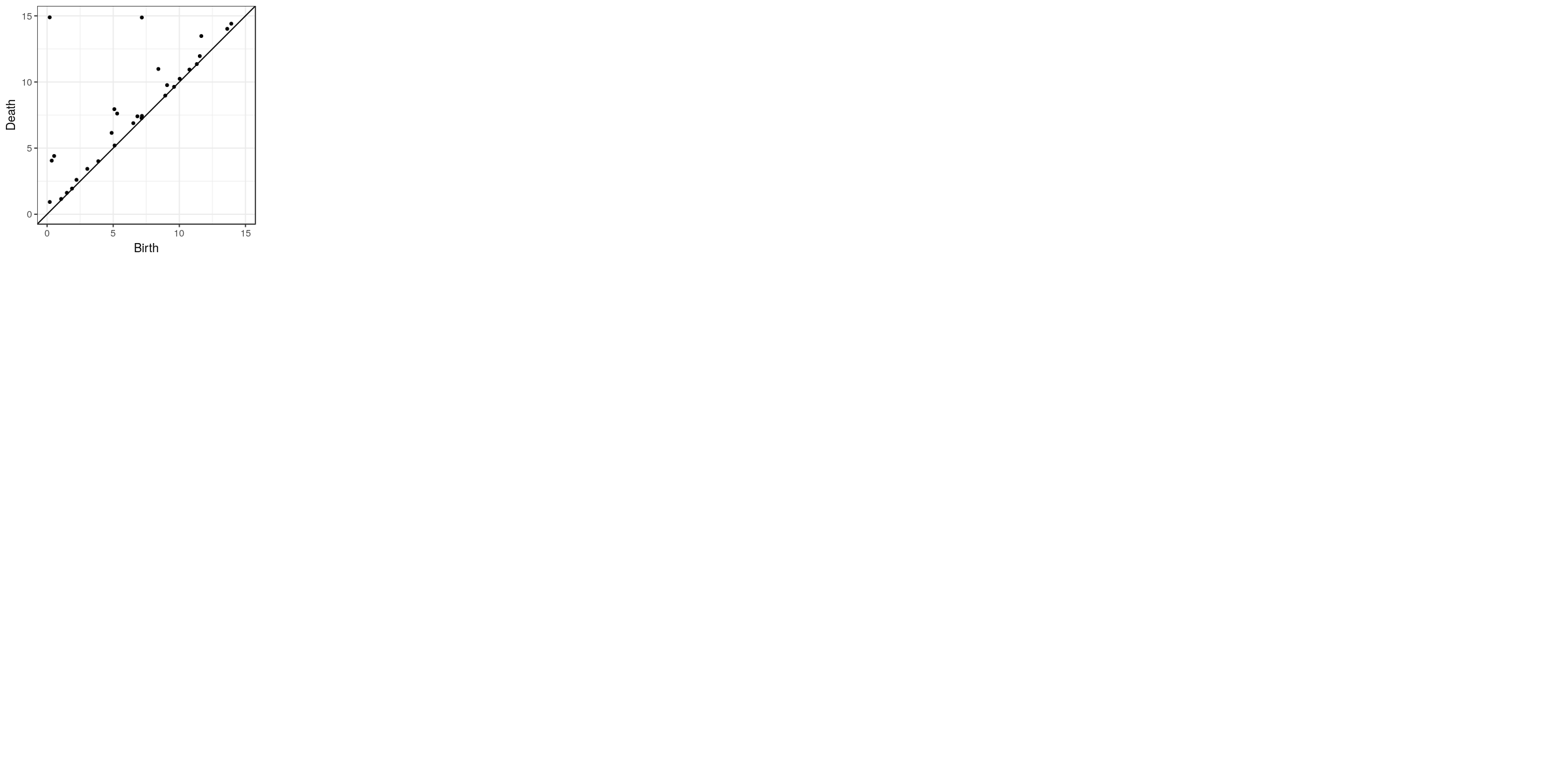}
 \caption{Persistence diagram for the directed spanning forest on a Poisson point process in a $15 \times 5$-window.}
 \label{pd_net_fig}
\end{figure}  
   
Next, we proceed from this exploratory analysis to potential applications of Theorem \ref{T:GaussianLimitPoisson1} in the context of goodness-of-fit tests. To that end, we introduce in Section \ref{ts_sec} a specific test statistic and look at its type 1 and 2 errors in Section \ref{gof_sec}.

\subsubsection{Test statistics}
\label{ts_sec}
 The functional central limit theorem shows that continuous functionals of the persistence diagram are asymptotically Gaussian. We illustrate this effect at hand of a specific scalar statistic derived from the persistence diagram also considered in the setting of $M$-bounded features (\cite{svane}).
 
 More precisely, we rely on the accumulated persistence function (\cite{biscioTDA}). That is, we aggregate the life times of all branches with birth times in a time interval $[0, r_{\ms L}]$ with $r_{\ms L} \le r_{\ms f}$:
$$\int_{[0, r_{\ms L}] \times [0, r_{\ms f}]} (d - b) \ms{PD}(\cP_n)({\rm d}b, {\rm d}d).$$
By the functional CLT, the statistic $T_{\ms L}$ is asymptotically Gaussian. We now illustrate that this distributional convergence becomes already clearly apparent on bounded sampling windows. To that end, we first compute the recentered and normalized statistics on the iid samples of the Poisson process. We compute the centered and normalized test statistics for 10,000 realizations of the Poisson model and then compare with the asymptotic Gaussian distribution. Then, Figure \ref{hist_cech_fig2} highlights that the resulting histogram is close to the density of a standard normal distribution. Also the Q-Q-plot in Figure \ref{hist_cech_fig2} supports this impression.

\begin{figure}[!htpb]
 \centering
 \includegraphics[trim={0cm 0cm 15cm 0cm}, clip, width=0.39\textwidth]{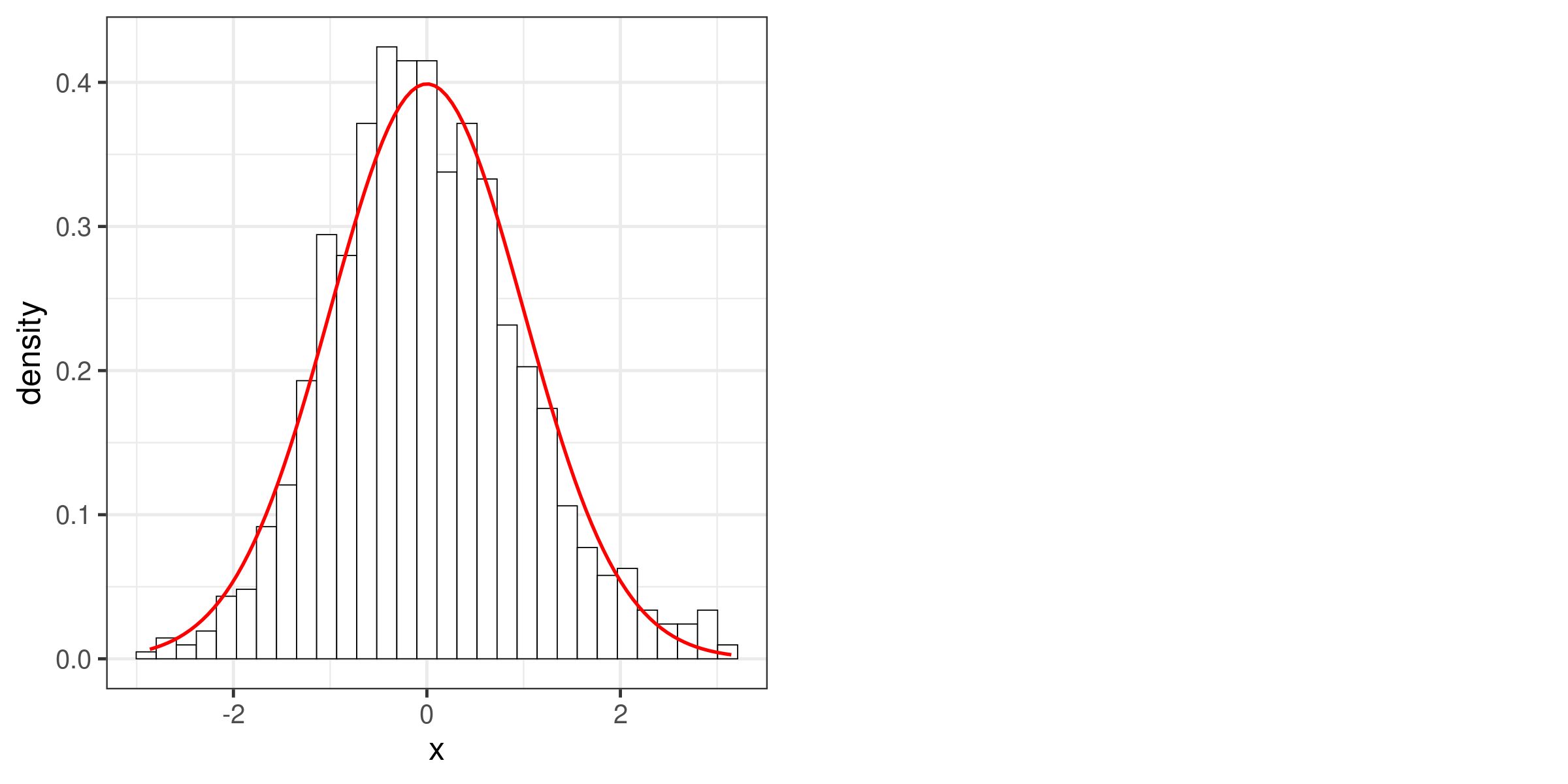}~\includegraphics[trim={0cm 0cm 15cm 0cm}, clip, width=0.39\textwidth]{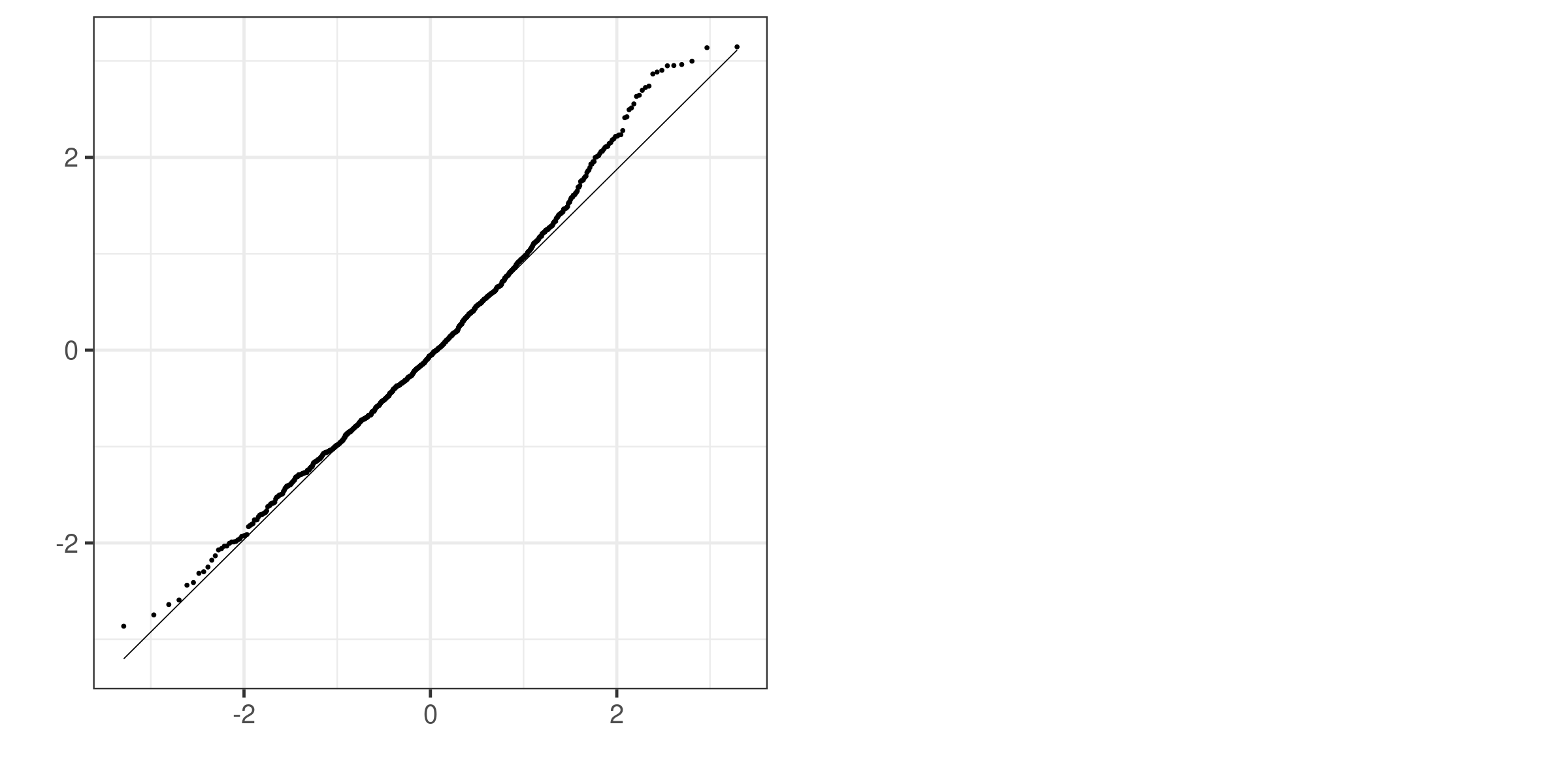}
 \caption{Histogram and Q-Q plot of normalized test statistics for the directed spanning forest.}
 \label{hist_cech_fig2}
\end{figure}

\subsubsection{Goodness-of-fit tests}
\label{gof_sec}
After illustrating that the test statistic $T_{\ms L}$ is asymptotically Gaussian, we now sketch how to derive goodness-of-fit tests. To that end, we extend the basic simulation setting from Section \ref{ts_sec} in two directions. First, we compute the TDA-based test statistic on Mat\'ern cluster and Strauss processes, which are point processes exhibiting a greater degree of clustering and repulsion, respectively. Second,  we also perform simulations if the intensity of the underlying point pattern is either 20\% lower or 20\% higher than the base case of intensity 2 from Section \ref{ts_sec}.

Our parameter choices for the attractive and repulsive point patterns are based on two principles. First, the intensity of these point patterns should match up with the intensity of the corresponding Poisson process. The intensity is the most fundamental characteristic of a point pattern and when fitting a model to data, it should reflect this property accurately. Second, to assess how strongly a TDA-based methodology can perform in challenging testing problems, we fixed parameters so that the resulting point patterns deviate only to a subtle degree from complete spatial randomness. To illustrate this point, Figure \ref{alt_fig} juxtaposes samples from the null model and samples from the alternatives.

\begin{figure}[!htpb]
 \centering
 \includegraphics[ width=0.99\textwidth]{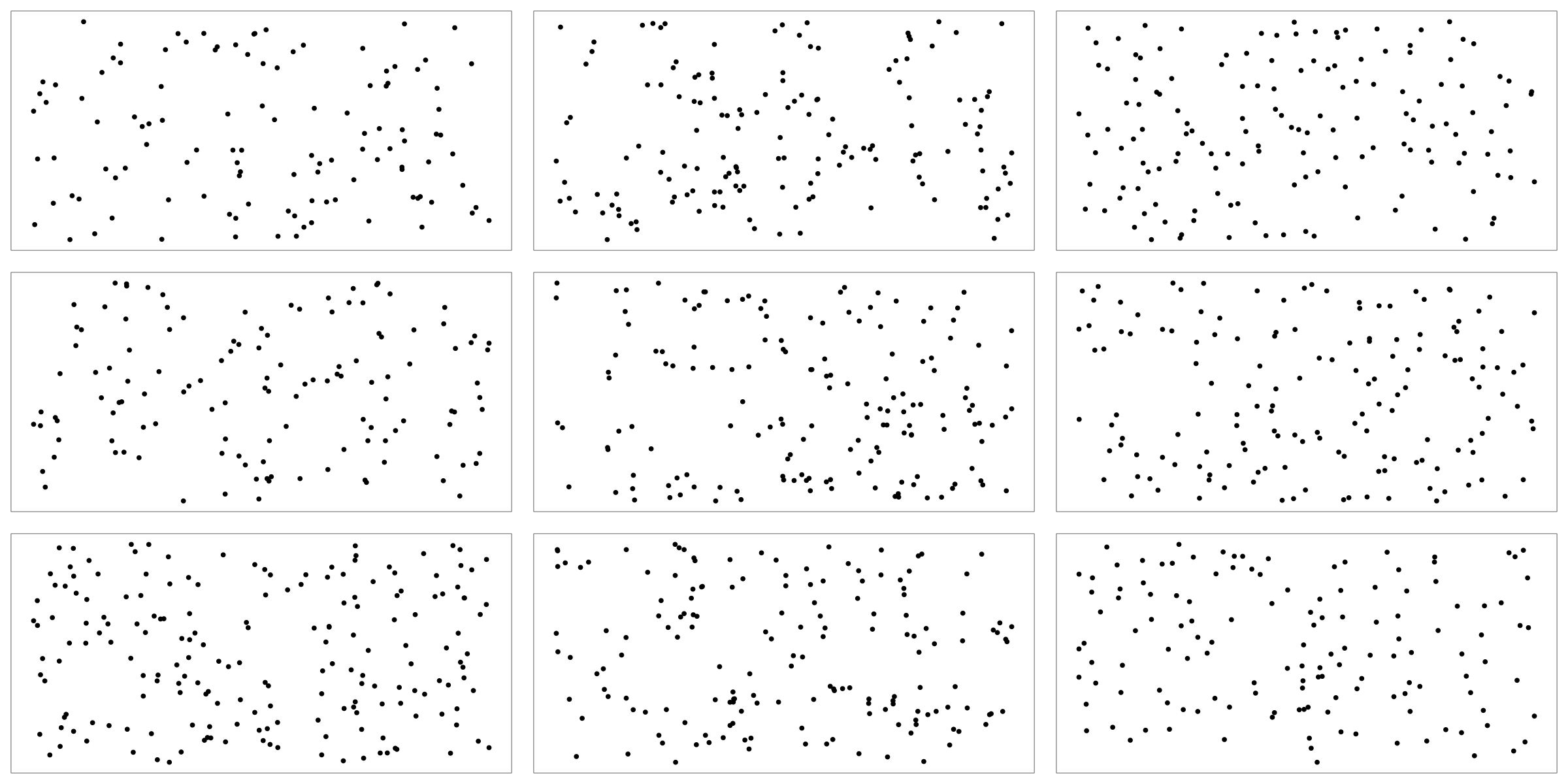}
 \caption{Samples from Poisson process (left), the Mat\'ern process (center) and the Strauss process (right). The intensity increases when moving from top to bottom.}
 \label{alt_fig}
\end{figure}

More precisely, we work with Mat\'ern cluster processes $\ms{MatC}(1.6, 0.5, 1)$, $\ms{MatC}(2, 0.5, 1)$ and $\ms{MatC}(2.4, 0.5, 1)$. 
 The first parameter indicates the intensity of the underlying Poisson parent process. Once this is fixed, we generate a $\ms{Poi}(1)$ number of offspring uniformly in a disk of radius $0.5$ around each parent. Concerning the Strauss process, we work with the parameter combinations $\ms{Str}(2.7, 0.6, 0.5)$, $\ms{Str}(4.0, 0.6, 0.5)$ and $\ms{Str}(5.4, 0.6, 0.5)$.  Here, the first parameter corresponds to the intensity of the associated interaction-free Poisson point process. Then, we fix the interaction parameter $\gamma$ and the interaction radius $r$ to be $0.6$ and $0.5$, respectively.

Table \ref{pow_net_tab} shows that when drawing 1,000 samples from the Poisson null model, the actual type 1 error is close to the nominal asymptotic 5 \% level. However, when moving to the alternatives, we see that the test power is in general rather small. Only for the Strauss process at low and moderate intensities the rejection rates depart noticeably from the given significance level. Hence, for the directed spanning forest, moving a bit away from a Poisson distribution of nodes induces only very subtle changes in the network structure that are difficult to detect with the persistence diagram. At first sight, this may come as a disappointment. However, we stress that network-based functional CLTs offer the potential to detect differences in the formation of the network topology, even if the underlying point pattern remains Poisson.

   \begin{table}[!htpb]
 \begin{center}
   \caption{Rejection rates for the test based on the sublevel filtration in the directed spanning forest.}
   \label{pow_net_tab}

 \begin{tabular}{rllc}
   & $\ms{Poi}$ & $\ms{MatC}$ & $\ms{Str}$ \\
\hline
   $\lambda = 1.6$   & 3.3\% & 5.3\% & 12.8\% \\
   $\lambda = 2.0$   & 4.9\% & 6.0\% & 7.8\% \\
   $\lambda = 2.4$   & 4.8\% & 6.3\% & 5.1\% \\
\end{tabular}
 \end{center}
   \end{table}

%
%
\subsection{\v Cech-filtration}
\label{cech_sec}
In Theorem \ref{T:GaussianLimitPoisson2}, we established a functional CLT that allows to deduce asymptotic normality for a wide variety of statistics derived from the persistence diagram. Now, we illustrate this behavior. To begin with, Figure \ref{cech_fig} highlights the persistence diagrams for 0- and 1-dimensional features associated with one realization of the Poisson point process.
\begin{figure}[!htpb]
 \centering
 \includegraphics[width=0.59\textwidth]{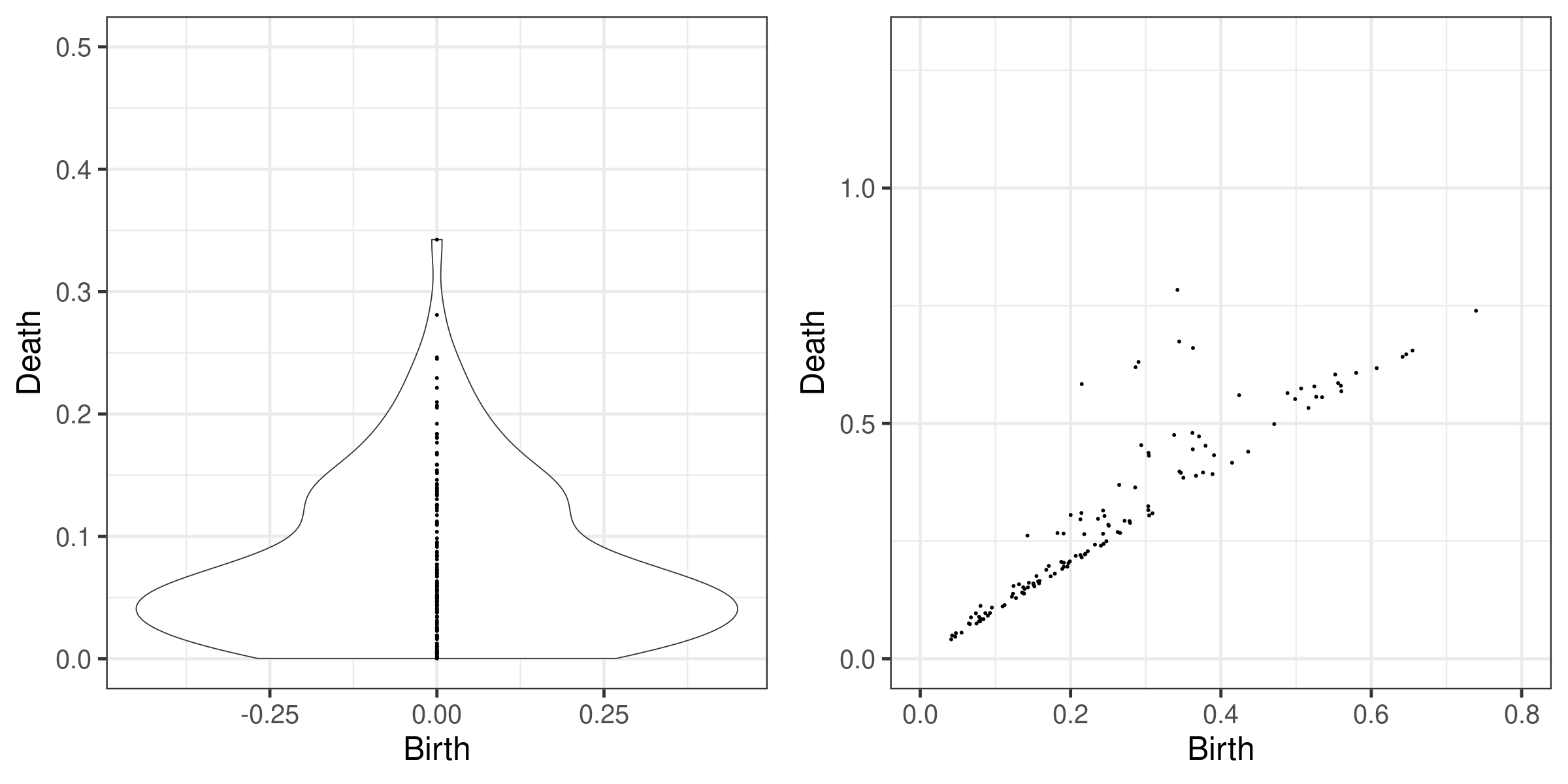}
 \caption{Persistence diagrams for 0-dimensional (left) and 1-dimensional (right) features of the \v{C}ech complex.}
 \label{cech_fig}
\end{figure}

\subsubsection{Test statistics}
Next, we introduce scalar test statistics derived from the persistence diagrams. For the 1-dimensional features, i.e., loops, we rely on the variant $T_{\ms L}$ of accumulated persistence function from Section \ref{ts_sec}. For features in dimension 0, we look at the integrated number of cluster deaths until a time $r_{\ms c}$, i.e., 
$$\int_0^{r_{\ms c}} \ms{PD}^0(\cP_n)([0, d]) {\rm d}d.$$
Since it is a continuous functional of the persistence diagram, it is asymptotically normal, so that we deduce from the functional CLT that it becomes Gaussian in large windows. When dealing with data, the intensity needs to be estimated and it was found in \cite{svane} that the following intensity-adapted variant leads to superior test powers:
\begin{align}
  \label{t0Def}
  T_{\ms C} := \frac1{\sqrt \lambda|W|}\int_0^{r_{\ms C} / \sqrt\lambda} \ms{PD}^0(\cP_n)([0, d]) {\rm d}d.
\end{align} 
Then, similarly to Section \ref{ts_sec}, the histogram and Q-Q plot in Figure \ref{hist_cech_fig} illustrate that the Gaussian approximation is already accurate in the moderately-sized window of the simulation study, even when relying on the more complex \v C-filtration.

For 1-dimensional features, we leverage the accumulated persistence function (\cite{biscioTDA}). That is, we aggregate the life times of all loops with birth times in a time interval $[0, r_{\ms L}]$ with $r_{\ms L} \le r_{\ms f}$:
$$\int_{[0, r_{\ms L}] \times [0, r_{\ms f}]} (d - b) \ms{PD}^1(\cP_n)({\rm d}b, {\rm d}d).$$
By the functional CLT, the cluster- and loop-based statistics $T_{\ms C}$ and $T_{\ms L}$ are asymptotically Gaussian. We now illustrate that this distributional convergence becomes already clearly apparent on bounded sampling windows. To that end, we first compute the recentered and normalized statistics on the iid samples of the Poisson process.

\begin{figure}[!htpb]
 \centering
 \includegraphics[width=0.59\textwidth]{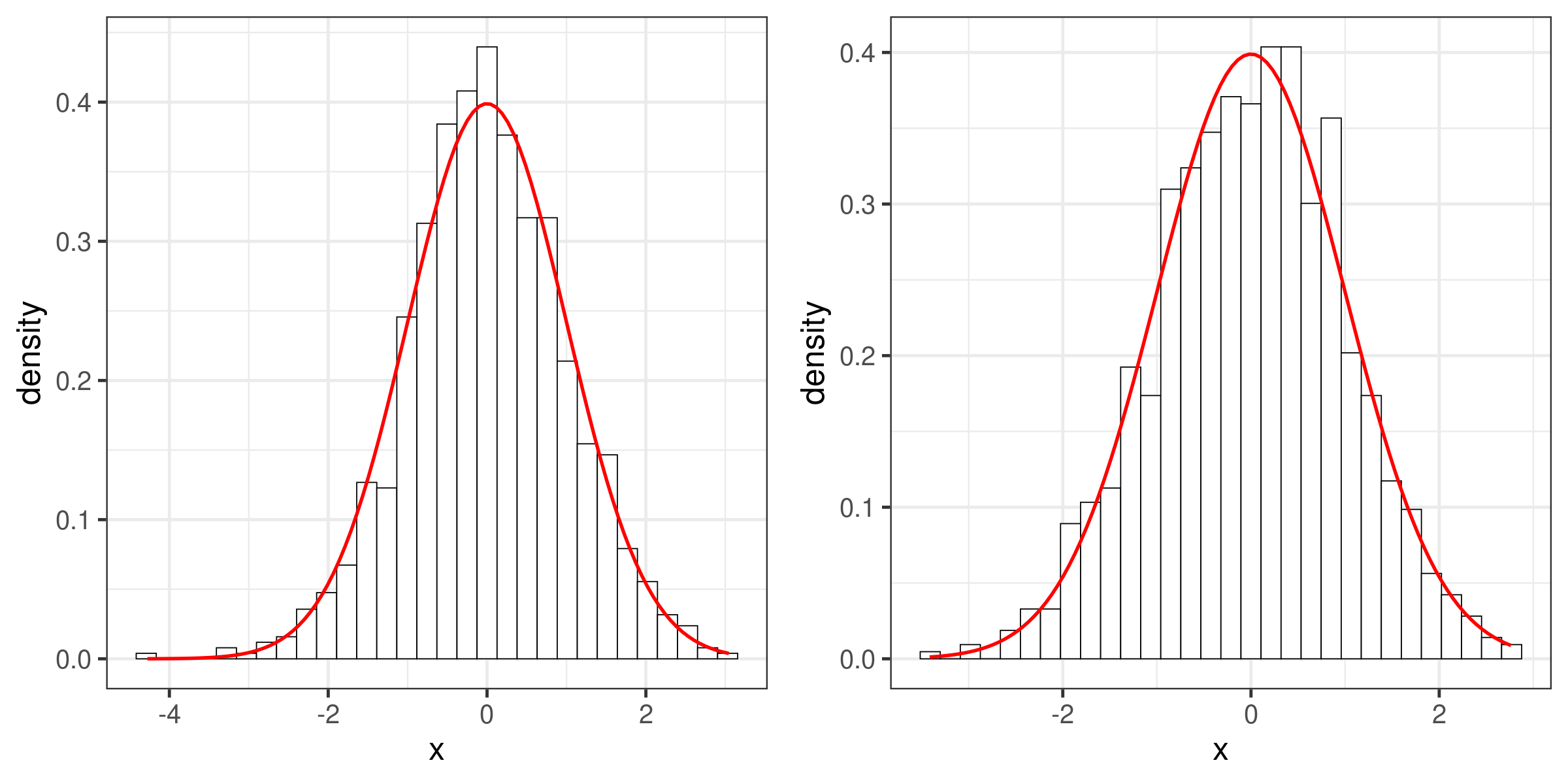}
 \includegraphics[width=0.59\textwidth]{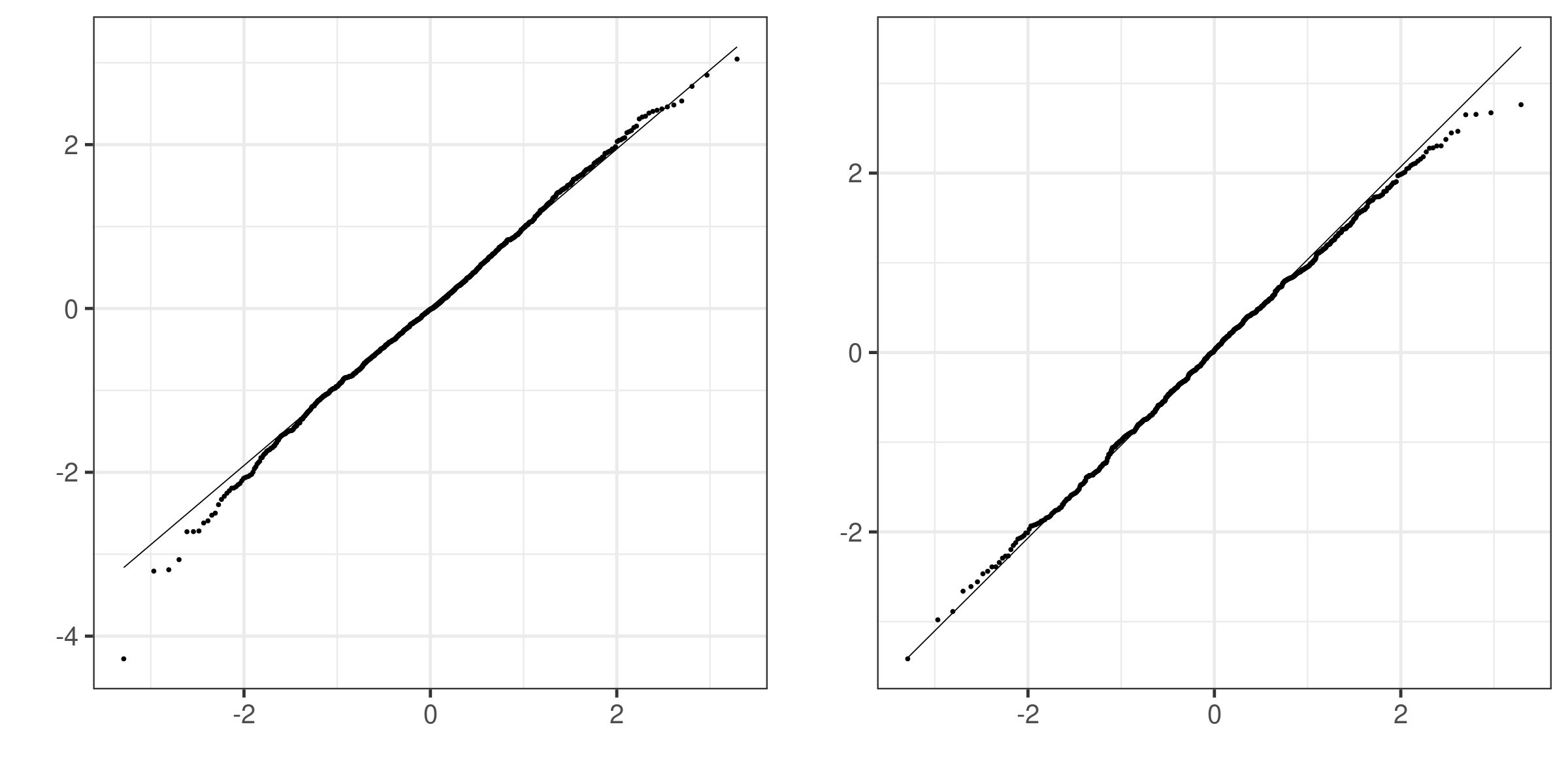}
 \caption{Histograms (top) and Q-Q plots (bottom) of normalized test statistics for 0-dimensional (left) and 1-dimensional (right) features against standard normal densities.}
 \label{hist_cech_fig}
\end{figure}
\subsubsection{Goodness-of-fit tests}
Again, similar to Section \ref{dsf_sec}, one possible application of the functional CLT would concern goodness-of-fit tests. To that end, replace the Poisson point process either by the Matérn cluster process $\ms{MatC}(2, 0.5, 1)$ or by the Strauss process $\ms{Str}(4.0, 0.6, 0.5)$. Again, we vary also the intensity.

Table \ref{pow_cech_tab} shows that when drawing 1,000 samples from the Poisson null model, the actual type 1 error is close to the nominal asymptotic 5 \% level. When moving to the type 2 errors, we see that the persistence-based tests have in general a reasonable power, in the light of the visual closeness of the underlying point patterns.  A more detailed inspections reveals noteworthy facets depending on the intensity that are particularly visible for the Strauss alternative.

Here, the performance of the cluster-based test statistics is sensitive to the intensity. For instance, in the high-intensity setting it achieves a power of $98.5\%$. This is consistent with the observation that at high intensities, the repulsion is strongly felt at a local level as it encourages strongly that points remain isolated for a certain amount of time before connecting to another component. On the other hand, the repulsion becomes more subtle when moving to lower intensities. Here, it appears that the effect is rather a greater degree of regularity on a larger scale, which can then be better detected by the loop-based test statistic. Finally, a seemingly paradoxical phenomenon is that the rejection-rate of the cluster-based test statistics at the low intensity is actually below the nominal level. We hypothesize that this could {be }due to the following reason. If the mean value under the null hypothesis and the alternative are very similar, then the higher degree of regularity of the Strauss process would lead to the test statistics falling outside the confidence band less often than it is the case for the more varying Poisson null model.

As a take-away, we recommend that one should in practice not rely on a single TDA-based test as the performance can vary substantially depending on the specific testing situation. This strengthens the case of having a fully-fledged functional CLT at one's disposal as this allows for a high degree of flexibility when it comes to constructing suitable test statistics.

   \begin{table}[!htpb]
 \begin{center}
	 \caption{Rejection rates for the test statistics $T_{\ms C}$ (left) and $T_{\ms L}$ (right) for the \v{C}ech complex.}
   \label{pow_cech_tab}

 \begin{tabular}{rllc}
   & $\ms{Poi}$ & $\ms{MatC}$ & $\ms{Str}$ \\
\hline
   $\lambda = 1.6$   & 5.4\% & 54.0\% & 2.3\% \\
   $\lambda = 2.0$   & 5.1\% & 46.8\% & 36.2\% \\
   $\lambda = 2.4$   & 4.1\% & 45.9\% & 98.5\% \\
\end{tabular}
\qquad 
 \begin{tabular}{rllc}
   & $\ms{Poi}$ & $\ms{MatC}$ & $\ms{Str}$ \\
\hline
   $\lambda = 1.6$   & 5.3\% & 40.5\% & 88.6\% \\
   $\lambda = 2.0$   & 4.4\% & 51.3\% & 43.7\% \\
   $\lambda = 2.4$   & 4.8\% & 50.5\% & 14.8\% \\
\end{tabular}
 \end{center}
   \end{table}

\section{Proof of Theorems \ref{T:GaussianLimitPoisson1} and \ref{T:GaussianLimitPoisson2} -- Tightness}
\label{proof_sec}

Proving a functional CLT involves two steps: (1) asymptotic normality of the multivariate marginals and (2) tightness. We note that the multivariate asymptotic normality was presented previously in a very similar context and holds in a more general setting than just for Poisson processes (\cite{hiraoka2018limit, krebs2019asymptotic}). In order to focus on the new information, we defer the corresponding proofs to Appendix~\ref{AppendixA}. 

Henceforth, we fix a value $1 \le q \le d - 1$ throughout the section. Many of the arguments do not differ at all between  $\b_{\cK, q, n}^{r,s}$ and $\bt^{r,s}$, so that we put $\bn^{r, s}$ to denote either of them. If there are differences, we point this out in detail. Furthermore, to ease notation, we write $\cP(W) := \#(\cP \cap W)$ for the number of points of $\cP$ in a domain $W \subset \R \times A$. We abbreviate constants whose value is not important by $c, c', C, C'$; so the values of $c, c', C, C'$ may change from line to line.

For one-dimensional c\`adl\`ag processes \cite{billingsley1968convergence} provides a highly convenient Chentsov-type moment condition. Since the persistent Betti numbers depend on two parameters, we will invoke the multi-dimensional extension in \cite{bickel1971convergence}  for c\`adl\`ag processes on $J$, where $J$ is either $J_{\cK}$ or $\Jt$ depending on the model. 
 Loosely speaking, \cite{bickel1971convergence} give a sufficient condition for tightness in terms of suitable moment estimates for the value of the process in blocks in $J$. Here,  a \emph{block} $E = \Eb \times \Ed $ is a product of intervals $\Eb, \Ed \subset [0, \infty)$ such that $\Eb \times \Ed \subset J$. For a function $f:\, J \to \R$ and a block $[b_-, b_+] \times [d_-, d_+]$, we define the increment
 $$ 	f([b_-, b_+] \times [d_-, d_+]):= f(b_+, d_+) - f(b_+, d_-) - f(b_-, d_+) + f(b_-, d_-).$$
 In particular, when applied to the setting of persistent Betti numbers, this statement reads as 
\[		 \b_n(E)= \b_n( E, \cP)= \#\{i\ge1:\, (B_i, D_i) \in E \},\]
where we consider all $q$-dimensional features $(B_i,D_i)$ obtained from the point cloud $\cP_n$ for a given filtration rule.
 
For the convenience of the reader, we now restate the condensed form of \cite[Theorem 2, Theorem 3]{bickel1971convergence} from \cite[Theorem 2]{lavancier}. To connect it to the specific problem of the present paper, we simplify to the two-dimensional setting and where the measure in the moment condition is a multiple of the Lebesgue measure, which we denote by $|\cdot|$ henceforth.
\begin{theorem}[Bickel \& Wichura]
	\label{bw_thm}
	For every $n \ge 1$, let $\{S_n(t)\}_{t \in J}$ be a c\`adl\`ag process whose finite-dimensional marginals converge to that of a process $\{X(t)\}_{t \in J}$. Suppose that $X(t)$ is continuous at the top-right corner of $J$ and that the finite-dimensional marginals of $S_n$ converge in distribution to those of $X(t)$. Suppose further that there exist $\e , C > 0$ such that 
	\begin{align}
		\label{lav_eq}
		\EE[S_n(E)^4] \le C |E|^{1 + \e}
	\end{align}
	holds for any block $E \subset J$. Then, in the Skorokhod topology, the process $\{S_n\}_{n \ge 1}$ converges in distribution to $X$.
\end{theorem}

 In fact, we rely on a very convenient variant of this condition mentioned in the remark following Theorem 3 in \cite{bickel1971convergence}, which allows to put a lower bound on the size of the blocks. This means we show that it suffices to consider blocks from the sequence of grids $(T_n)_n$ with { $T_n = \{ (i/[ n^{\a}]-1/2, kT/[n^{\a}]): [n^{\a}]\ge i, k\ge 0 \}$ for the directed filtration and $T_n = \{ (iT/[ n^{\a}], kT/[n^{\a}]): [n^{\a}]\ge i, k\ge 0 \}$ for the \v C- and VR-filtration}, $n \ge 1$, where we henceforth fix $\a = 3/4$.
To state this variant precisely, for $\de > 0$ we first recall the modulus of continuity of a multi-parameter c\`adl\`ag function $z\colon [0, T]^2\rightarrow \R$
\[
			\omega'_\de(z) := \inf_\Gamma \max_{G\in\Gamma} \sup_{s, t\in G} |z(t)-z(s)|, 
\]
where the infimum extends over all $\de$-grids $\Gamma$ in $J$. 
Now, we proceed essentially in two steps as follows:

 (1) We show in Proposition~\ref{grid_prop} that it suffices to compute the modulus of continuity on a sequence of equidistant grids $(T_n : n \ge 1)$ whose union lies dense in $J$. So instead of considering the modulus of continuity of the centered process
 $$\bar\bn^{r, s} := \bn^{r, s} - \EE[\bn^{r, s}],$$ 
 we study the modulus of continuity of the process $\bar\bn$ restricted to $T_n$, i.e, $\bar\bn |_{T_n}$. 
 
 (2) We verify the moment condition of \cite{bickel1971convergence} for the process $\bar\bn |_{T_n}$.
 
\noindent\textit{Step 1 -- Reduction to a grid.} 
To reduce to a grid, we rely on two auxiliary results from literature. First, we recall the definition of the \emph{filtration time} 
$$r(\s):=\inf\{r > 0:\, \s \in \cK_r\}$$
as the first time $r > 0$ where a simplex $\s \in \cK_T$ is contained in the complex $\cK_r$. In \cite[Lemma~6.10]{divol2020} it is shown that both in the \v C- and the VR-filtration, the filtration times of $p \ge 2$ iid uniform points in a ball admit a bounded and continuous density on $\R_+$, which is bounded above by some $c_{1, d} > 0$.

The second auxiliary result is the Slivnyak-Mecke formula \cite[Theorem 4.4]{lastPenrose} -- a highly flexible tool for computing expectations with respect to a Poisson point process. To make the presentation more accessible, we extend the Poisson point process $\cP$ to the entire space $\R^d$ and state the Slivnyak-Mecke formula for this specific case.
\begin{theorem}[Slivnyak-Mecke]
	\label{smf}
	Let $q \ge 0$ and $f:\, \R^{d(q + 1)} \to [0, \ff)$ be measurable. Then,
	\begin{align}
		\label{meck_eq}
		\EE\Big[\sum_{\substack{Z_{i_0}, \dots, Z_{i_q} \in \cP \text{ pw. distinct}}}f(Z_{i_0}, \dots, Z_{i_q})\Big] = \int_{\R^{d(q + 1)}} f(z_0, \dots, z_q) \d(z_0, \dots, z_q).
	\end{align}
\end{theorem}
Note that if $f$ is supported on $M^q$ for some compact $M \subset \R^d$, then the right-hand side in \eqref{meck_eq} can be expressed as $|M|^{q + 1}\EE[f(Z'_0, \dots, Z'_q)]$ with the $Z'_0, \dots, Z'_q$ iid.~uniform on $M$.

\begin{proposition}[Reduction to grid]
\label{grid_prop}
Let $\e, \de >0$. Then, there exists a deterministic $n_0 = n_0(\e, \de)\in\N$ such that almost surely
	$$\sup_{n \ge n_0}n^{-1/2}|\omega'_\de(\bar\b_n)- \omega'_\de(\bar\b_n|_{T_n})| \le \e.$$
\end{proposition}
\begin{proof}
We split the proof into two parts, the first one considering the case of a directed filtration, the second one the case of the \v C- and VR-filtration.

	\textit{Directed filtration}. The process $\bt^{r, s}$ is increasing in $r$ and $(-s)$. Hence, by \cite[Corollary 2]{davydov}, it suffices to take {$0\le r_2 - r_1 \le 1/[n^\a]$ and $r, r_1, r_2 \in [-1/2, 1/2]$ as well as $0\le s_2 - s_1 \le T/[n^{\a}]$ and $s, s_1, s_2 \in [0, T]$} and then bound the differences
\[
	n^{-1/2} \ \E{ \bt^{r_2, s} - \bt^{r_1, s} }\quad \text{ and }\quad n^{-1/2} \ \E{ \bt^{r, s_2} - \bt^{r, s_1}}.
\]

We explain how to proceed for the second claim, noting that the steps for the first are similar, but easier. By Palm theory for the Poisson point processes, it suffices to show that 
	$$\lim_{n \to \infty}\sup_{\substack{z \in W_n}}\sqrt n p_n^z([ s_1, s_2])=0, $$
	where $p_n^z([ s_1, s_2])$ denotes the probability that an additional branch born at location $z$ in the network dies in the time interval $[s_1, s_2]$. A necessary condition for a branch to die in that time interval is that $\cP \cap ([s_1, s_2] \times A) \ne \es$. Since, $\cP$ is a Poisson point process with unit intensity, we therefore obtain that 
	$$ \sup_{\substack{z \in W_n}} p_n^z([ s_1, s_2]) \le \E{\cP([s_1, s_2] \times A)} \le c |A| n^{-3/4} = o(n^{-1/2}), $$
as asserted.

\noindent\textit{\v C- and VR-filtration}.
Let {$0\le r_2 - r_1 \le T/[n^{\a}]$ and $0\le s_2 - s_1 \le T/[n^{\a}] $ and $r, s, r_1, s_1, r_2, s_2 \in [0,T]$}. We study the differences
\[
	n^{-1/2} \ \EE\big[ \b^{r_2, s}_{\cK, q, n} - \b^{r_1, s}_{\cK, q, n} \big] \quad\text{ and }\quad n^{-1/2} \ \EE\big[ \b^{r, s_2}_{\cK, q, n} - \b^{r, s_1}_{\cK, q, n }\big]
\]
and only tackle the first expression, as the arguments for the second are very similar. Since we deal with the \v C- and VR-filtration, we obtain from \eqref{Def:PersistentBetti0} that
\begin{align}\begin{split}\label{E:ModCont1}
\b_{\cK, q, n}^{r_2, s} - \b_{\cK, q, n}^{r_1, s}&\le \dim \frac{Z_q(\cK_{r_2}(\cP_n) ) }{Z_q(\cK_{r_1}(\cP_n) )} - \dim \frac{Z_q(\cK_{r_2}(\cP_n) ) \cap B_q( \cK_s(\cP_n)) }{Z_q(\cK_{r_1}(\cP_n) ) \cap B_q( \cK_s(\cP_n)) } \\
		&= \dim \frac{Z_q(\cK_{r_2}(\cP_n) ) + B_q( \cK_s(\cP_n)) }{Z_q(\cK_{r_1}(\cP_n) ) + B_q( \cK_s(\cP_n)) } \\
		&\le \# \{ q-\text{simplices } \s\in \cK_T(\cP_n): r_1 \le r(\s) \le r_2 \},
\end{split}\end{align}
where	we have used for the second equality, the dimension formula $\dim U \cap V = \dim U + \dim V - \dim (U\oplus V)$, for two subspaces $U, V$ of a vector space $W$.  In view of leveraging the density for the filtration times from \cite[Lemma~6.10]{divol2020}, we let $M$ be a sufficiently large ball containing	$[- T , 1 + T ] \times A$. We write $\th_i(x, y) = (x - i, y)$ for the shift of a point $(x, y) \in \R \times A$ by $i \in \R$ along the $x$-axis and we extend this definition in the natural way to describe a shift of sets along the $x$-axis. A corresponding shift $M_i := \theta_{-i}(M)$ of $M$ then covers the shifted window $[-T + i, 1 + T + i]$. We write $\cP_i^+$ for the unit-intensity Poisson point process on $M_i$.

	With this extension, we now bound the expected difference  $\EE\big[\b_{\cK, q, n}^{r_2, s} - \b_{\cK, q, n}^{r_1, s}\big]$ by applying the Slivnyak formula with
	$$f(z_0, \dots, z_q) = 1\big\{r(z_0, \dots, z_q) \in (r_1, r_2]\big\} 1\big\{z_0, \dots, z_q \in M_i\big\}.$$
	Thus, letting $Z_0', \dots, Z_q'$ be iid {with a uniform distribution } on $M$ gives that
\begin{align*}
	\EE[ \b^{r_2,s}_{\cK,q,n}- \b^{r_1,s}_{\cK,q,n} ]	&\le  \sum_{|i|\le n/2} \EE \big[\# \{ q-\text{simplices } \s\in \cK_T(\cP_n): r_1 \le r(\s) \le r_2,\, \s \cap ([i ,i+1]\times A) \ne \es \} \big] \\
	&\le  \sum_{|i|\le n/2} \EE \big[\# \{ q-\text{simplices } \s\in \cK_T(\cP_i^+): r_1 \le r(\s) \le r_2\} \big] \\
			&\le \sum_{|i|\le n/2}|M_i|^{q + 1}\P\big(r_1 \le r( \{Z_0',\dots,Z_q'\} ) \le r_2\big).
			 \end{align*}
			 Thus, using the aforementioned result from \cite[Lemma~6.10]{divol2020} that filtration times admit a bounded density, we see that $\EE[ \b^{r_2,s}_{\cK,q,n}- \b^{r_1,s}_{\cK,q,n} ]$ is bounded above by $c' n (r_2 - r_1)$ for a suitable $c' > 0$. This completes the proof.
 \end{proof}

\noindent\textit{Step 2 --The moment condition}
The second step consists in deriving the Chentsov-type moment condition in \eqref{lav_eq} provided that the blocks are taken from the grid.

Proceeding in the vein of \cite{penrose2001central}, we first discretize the space into intervals and then leverage a martingale-difference decomposition. Then, the decomposition takes the basic form,
\begin{align}
	\label{bnDec}
\bar\b_n(E) = \sum_{|i|\le \kn} D_{i, n}(E).
\end{align}
To define the increments $D_{i, n}(E)$ precisely, we introduce for $i \in Z$
$$
	\cG_{i} := \s\big(\cP \cap \big( (-\infty, i + 1/2] \times A \big)\big)
$$
as the $\s$-algebra of the information coming from the configuration of $\cP$ in $(-\infty, i + 1/2] \times A$.  If $n$ is odd, then $\bn(E) = \EE[\bn(E) | \cG_{ (n-1)/2} ]$ and $\EE[\bn(E)] = \EE[ \bn(E)| \cG_{- (n-1)/2 - 1} ]$, where the last equality follows because $\cG_{- (n-1)/2 - 1}$ and $\sigma(\bn(E))$ are independent.
If $n$ is even, then $\bn(E) = \EE[\bn(E) | \cG_{n/2}]$ and $\EE[\bn(E)] = \EE[ \bn(E)| \cG_{-\kn - 1} ]$.
We set
\begin{align}
	\label{dnDef}
	D_{i, n}(E) := \EE\big[\bn(E)|\cG_{i}\big] - \EE\big[\bn(E)|\cG_{i - 1}\big]
\end{align}
and \eqref{bnDec} is satisfied. 
A delicate part in bounding the right-hand side of \eqref{bnDec} are upper bounds on the $|D_{i,n}(E)|$, which we derive in Lemma \ref{diffBoundLem} below. To prove this result, we rely on a key advantage of working with the Poisson point process: we can re-express $D_{i, n}$ in terms of a single conditional expectation with respect to $\cG_{i}$. 

To make this precise, let $\cP'$ be an independent copy of the Poisson process $\cP$ and write $\cP'_n$ for the restriction of $\cP'$ to $W_n$. 
Now, we obtain for $i \in \R$ 
\begin{align}\label{Poisson_copy}
	\cP_{i, n} := [\cP_n\setminus \th_{-i}(W_1) ] \cup [ \cP_n' \cap \th_{-i}(W_1)].
	\end{align}
and denote by $\bnz(E) = \bn(\cP_{i, n})(E)$ the increment computed on the basis of $\cP_{i, n}$ instead of $\cP_n$. Then, $\EE\big[ \bn(E) | \cG_{i-1} \big] = \EE\big[ \bn(\cP_{i, n} )(E) | \cG_{i} \big]$ yields that
$
   D_{i, n}(E) =  \ \EE\big[\bn(E) - \bnz(E)\ | \ \cG_{i} \big].
   $

We rely essentially on two ingredients, namely on stabilization and on moment properties of the Poisson point process. As a preliminary observation, we note that by Jensen's inequality, it suffices to derive bounds for
$$\dnz(E) := \bn(E) - \bnz(E).$$

To that end, we need a form of external stabilization. First, we discuss the directed filtration.
For $j \in \Z$ let $R_j$ denote the radius of stabilization for the shifted process $\th_j(\cP)$. Furthermore, put 
$$A_i^+ := \sup\{j \ge i:\, \th_{-i}(W_1) \cap \th_{-j}(W_{R_j}) \ne \es\}$$
as the last index to the right of $i$ that can still be influenced by changes at $\th_{-i}(W_1)$. Note that $A_i^+$ is almost surely finite, since the radii of stabilization have exponential tails. Similarly, we define 
$$A_i^- := \inf\{j\le i:\, \th_{-i}(W_1) \cap \th_{-j}(W_{R_j}) \ne \es\}$$
and put
\begin{align}
  \label{ext_stab_def}
  R_i' := 2T + 2\max\{i - A_i^-, A_i^+ - i\}.
\end{align}
Hence, switching from $\cP_n$ to $\cP_{i, n}$ does not influence life times of branches starting outside $\th_{-i}(W_{R_i'})$.

For the \v C- and VR-filtration, we can proceed along the same lines. More precisely, writing $\ms{Comp}(\cP)$ for the family of connected components of the \v C-resp. VR complex at parameter $T$, we set
\begin{align}
	\label{cech_stab}
	R_j := R_j' := 2T + \max_{\cC \in \ms{Comp}(\cP):\, \cC \cap \th_{-j}(W_1) \ne \es}2\ms{diam}(\cC)
\end{align}
as $2T$ plus the maximum of the diameter of the connected components intersecting $\th_{-j}(W_1)$.
Here, the exponential decay relies critically on the assumption that we consider the Poisson point process in the cylindrical tube $\R \times A$, where no percolation occurs.

%
%

In the proof of Proposition \ref{grid_prop}, we have already used \cite[Lemma~6.10]{divol2020} in order to derive a continuity property for the filtration time associated with a single simplex. In order to derive the moment bounds, we are facing a more complex task since the corresponding estimates involve now pairs of simplices corresponding to the birth- and death times.  We will therefore need a related property which concerns not just a single simplex but a pair of simplices. At the same time, we do not need that the corresponding density is bounded, but for us a H\"older-type property is sufficient.

\bepr[H\"older continuity of filtration times]
\label{Cech_prop}
	Let $-1\le q'  < q  + 1\le d$ and $\rho > 0$. 	Let $Z_0,\dots, Z_{q + 1}$ and $Y_{q'+1}, \dots, Y_q$ be iid uniform on $B(0, \rho)$ and write $\s = \{Z_0,\dots,Z_{q'}, Y_{q' + 1}, \dots,  Y_q\}$ and $\wt \s = \{ Z_0,\dots, Z_{q + 1} \}$.
	 	Then, there is a constant $C = C_{d,T} < \infty$ such that in the \v C-or VR-filtration for every block $E \subset J$,
\begin{align}\label{E:ContinuityC0}
	\p\big( (r(\s), r(\wt \s)) \in E,\,  r(\wt \s) > r(\s) \big) 	\le C_{d,T} |E|^{3/4}.
\end{align}
\enpr

In the VR-filtration, Proposition \ref{Cech_prop} could be sharpened by replacing $|E|^{3/4}$ with $|E|$. However, in the \v C-filtration, the geometry is substantially more involved.

\begin{proof}
	We only prove Proposition \ref{Cech_prop} for the \v C filtration as the arguments for the VR filtration are much simpler and an immediate consequence. In the following, we will bound above the filtration times by $T$ if this is useful and do not mention it any further. For $i \le q$ we set $Z_{\le i} := (Z_0, \dots, Z_i)$.

	We make use of the following two structural results laid out in detail in \cite[Lemma~6.10]{divol2020}. {We quote these results:} (i) The filtration time $r_q$ admits a continuous density on $\R_+$, which is bounded above by some $c_{1, d} > 0$.
	(ii) Almost surely,
	$$
		\p( r(Z_{\le i}) \in (s,t] \ba Z_{\le {i - 1}} ) \le c_{2,d} \int_{s \vee  r(Z_{\le i - 1}) }^t \frac{v^d  }{ \sqrt{ v^2 - r(Z_{\le i - 1} )^2 } } \intd{v},
		$$
	where $c_{2,d}\in\R_+$ depends only  on $d$, see \cite[Proof of Lemma~6.10]{divol2020}.
		Therefore, also for every measurable $g\colon [0, \infty) \to [0, \infty)$, almost surely,
		\begin{align}
			\label{cond_filt_eq}
					\EE[ g(r(Z_{\le i}))  \ba Z_{\le {i - 1}} ] \le c_{2,d} \int_{r(Z_{\le i - 1})}^\infty g(v) \frac{v^d  }{ \sqrt{ v^2 - r(Z_{\le i - 1} )^2 } } \intd{v}.
		\end{align}

	\textit{Part 1. $q' < q$.} In this case, we are in the situation, where the simplex $\wt\sigma$ contains at least two points which are not contained in the simplex $\sigma$.  Hence, applying the bound \eqref{cond_filt_eq} twice,
\begin{align*}
	\p( r(\wt \s) \in (s,t] \ba Z_{\le q - 1} )  &= \E{ \E{ \1{r(\wt \s ) \in (s,t] } \ba Z_{\le q } } \ba Z_{\le q - 1 }}  
	\le \EE\Big[ c_{2,d} \int_{s \vee r(Z_{\le q})}^t \frac{v^d  }{\sqrt{v^2 - r(Z_{\le q})^2} } \ \intd{v} \ \Big| \  r(Z_{\le q-1}) \Big]  \\
	&\le c_{2,d}^2 \int_{r(Z_{\le q - 1})}^t \int_{s \vee u}^t  \frac{u^dv^d  }{{\sqrt{v^2 - u^2} }\sqrt{u^2 - r(Z_{\le q -1})^2} } \ \intd{v} \ \intd u  \le c_{2,d}^2 \int_s^t v^d (2T^{d-1}) \ \intd{v},
\end{align*}
where the last inequality is derived as follows: Let $m=\sqrt{(v^2+r(Z_{\le q - 1})^2)/2}$, then
\begin{align*}
	\int_{r(Z_{\le q - 1})}^v \frac{ u }{ \sqrt{ v^2 - u^2 } \sqrt{ u^2 - r(Z_{\le q - 1})^2 } } \ \intd u  
	&\le \int_{r(Z_{\le q - 1})}^m \frac{ u }{ \sqrt{ v^2 - m^2 } \sqrt{ u^2 - r(Z_{\le q - 1})^2 } } \ \intd u  + \int_m^v \frac{ u }{ \sqrt{ v^2 - u^2 } \sqrt{ m^2 - r(Z_{\le q - 1})^2 } } \ \intd u \\
	&= \frac{\sqrt{ m^2 - r(Z_{\le q - 1})^2 }}{\sqrt{ v^2 - m^2 }} + \frac{\sqrt{ v^2 - m^2 }}{\sqrt{ m^2 - r(Z_{\le q - 1})^2 }} = 2.
\end{align*}
Consequently, almost surely, $\p( r(\wt \s) \in (s,t] \ba Z_{ \le q - 1} ) \le C (t-s)$. Therefore, writing $E = \Eb \times \Ed$,
\begin{align*}
	& \p( r(\s)\in \Eb, 	r(\wt \s)\in \Ed ) 
	= \EE\big[1\{r(\s) \in \Eb\}
	\p\big( r(\wt \s )\in \Ed  \ba Z_{\le q - 1}\big)\big	] \le C \ \p( r(\s ) \in \Eb) \ |\Ed| \le C'  |E|,
\end{align*}
where the last inequality follows because the unconditional density function is bounded. This completes the proof of part 1.

	\textit{Part 2. $q' = q$.}
First, we split the block $E$ into sub-blocks $E_1,\ldots,E_4$ such that the diagonal is contained in at most one block. Note that depending on the position of $E$ some blocks can be empty. Once, we have verified \eqref{E:ContinuityC0} for each sub-block, the claim follows by the additivity of the probability measure and because $\sum_{i \le 4} |E_i|^{3/4} \le 4 |E|^{3/4}$.

	It is clear, that sub-blocks below the diagonal will not contribute to the probability in \eqref{E:ContinuityC0}. To this end, we represent the block as $E = \Eb \times \Ed$ and only have to consider two cases: (a)$\Eb = \Ed$  and (b) $\Eb \cap \Ed = \es$.

	To deal with case (a),  we write $\Eb = \Ed = (a, b]$ and conclude  from \eqref{cond_filt_eq} that almost surely   on the event $r(\s) \in \Eb$,
\begin{align*}
	\p\big( r(\wt \s) \in \Ed  \ba Z_{\le q} \big) &\le 
					c_{2,d} \int_{\Ed } \frac{v^d1\{v >  r(\s)\}  }{ \sqrt{ v^2 - r(\s )^2 } } \intd{v}
					 \le c_{2,d}T^{d -1} \int_{\Ed} \frac{v 1\{v >  r(\s)\} }{ \sqrt{ v^2 - r(\s )^2 } } \intd{v}
					=c_{2,d}T^{d - 1} \sqrt{b^2 -  r(\s)^2}\le C T^{d - 1/2}\sqrt{|\Eb|}.
		\end{align*}
Hence, 
\begin{align*}
	& \p( r(\s)\in \Eb, 	r(\wt \s)\in \Eb ) 
	= \EE\big[1\{r(\s) \in \Eb\}
	\p\big( r(\wt \s )\in \Eb  \ba Z_{\le q }\big)\big	] \le CT^{d - 1/2}\p( r(\s ) \in \Eb)\sqrt{|\Eb|} \le C'_T  |E|^{3/4},
\end{align*}
where we rely again on the boundedness of the unconditional density.

It remains to deal with case (b). Here, we write $\Ed = (c, d]$ and again apply \eqref{cond_filt_eq} to obtain that conditioned on $r(\s) \in \Eb$,
\begin{align*}
	\p\big( r(\wt \s) \in \Ed  \ba Z_{\le q} \big) &\le c_{2,d}
					 \int_{\Ed } \frac{v^d  }{ \sqrt{ v^2 - r(\s )^2 } } \intd{v}
					 \le c_{2,d}T^{d -1} \int_{\Ed} \frac{v  }{ \sqrt{ v^2 - r(\s )^2 } } \intd{v}\\
					 &=c_{2,d}T^{d - 1} \Big(\sqrt{d^2 - r(\s)^2} - \sqrt{c^2 - r(\s)^2}\Big).
		\end{align*}
Now, by completing the square,
$$\sqrt{d^2 - r( \s)^2} - \sqrt{c^2 - r( \s)^2} =  \frac{d^2 - c^2}{\sqrt{d^2 - r( \s)^2} + \sqrt{c^2 - r( \s)^2}} \le \frac{d^2 - c^2}{\sqrt{d^2 - r( \s)^2}} \le 2\sqrt T\frac{d - c}{\sqrt{d - r( \s)}}.$$
If $|\Eb| \le |\Ed|$, then using that $r(\s)$ has a bounded density and that $\sqrt{d - u} \ge \sqrt{d - c}$ for $u \in \Eb$, we arrive at
$$\p\big( r(\s) \in \Eb, r(\wt \s) \in \Ed\big)\le 2c_{1,d}T^{d - 1/2} (d - c) \int_{\Eb} \frac1{\sqrt{d - u}}\d u \le 4c_{1,d}T^{d - 1/2} \sqrt{d - c}|\Eb| \le C'_T|E|^{3/4}.$$
On the other hand, if $|\Eb| \ge |\Ed|$, then we may again complete the square to obtain that
$$\p\big( r(\s) \in \Eb, r(\wt \s) \in \Ed\big)\le 2c_{1,d}T^{d - 1}(d - c) \int_{\Eb} \frac1{\sqrt{d - u}}\d u =4c_{1,d}T^{d - 1} (d - c)\big(\sqrt{d - a} - \sqrt{d - b}\big) \le C_T (d - c)\frac{b - a}{\sqrt{d - a}}.$$
Finally, we conclude the proof by noting that
$(d - c)({b - a})/{\sqrt{d - a}} \le (d - c) \sqrt{b - a} \le C'_T |E|^{3/4}.$
\end{proof}

We stress that the key ingredient for the moment bounds is the exponential decay of the radius of stabilization rather than the cylindrical setup. In particular, the arguments would work in the entire $\R^d$ if we restrict to the sub-critical regime of continuum percolation, or if we work with bounded features as in \cite[Lemma 9.6]{svane}.
\begin{lemma}[Moment bound]
	\label{diffBoundLem}
	For every $k \ge 1$ there exists a constant $C_k > 0$ such that for every $n \ge 1$ and every block $E = \Eb \times \Ed \subset J$, 
\begin{enumerate}
  \item { in the directed filtration, 
  $\max_{|i| \le \kn} \ \E{|\D_{i, n}(E)|^k} \le C_kn^{-16} + C_k\one\{i/n \in 2\Eb\}|\Ed|^{7/8}.$}
\item in the \v C- and VR-filtration,
	$\max_{|i| \le \kn} \ \E{|\D_{i, n}(E)|^k} \le C_k |E|^{11/16}.$
	\end{enumerate}
\end{lemma}

%
%
\begin{proof}
	First, we present arguments that are largely common to all of the three filtrations, namely an upper bound of the form
	\begin{align}
		\label{dnz_eq}
		|\dnz(E)| \le \De_{i,R_i'}'(E, \cP_n) + \De_{i,R_i'}'(E, \cP_{i, n}).
	\end{align}
	Here, for the directed filtration, $\De_{i, R_i'}'(E, \cP_n)$ denotes the number of branches in the network on $\cP_n$ with life time in $\Ed$ born in the domain $W_{i, R_i'} := \theta_{-i}(W_{R_i'})$. 

	In the \v C- and VR-setting, $\dnz(E)$ is determined by the features 
	contained in the filtration $\cK_T( \cP_n)$ which are not in $\cK_T( \cP_{i,n} )$ and vice versa. Hence, these features all have at least one simplex which intersects with $[i - 1/2, i + 1/2]\times A$. Therefore, both the $q$-simplex $\s \in \cK_T(\cP_n)$ giving birth to that feature, as well as the $(q + 1)$-simplex $\wt \s \in \cK_T(\cP_n)$ leading to the death of the feature are contained in the corresponding connected component, which is a subset of $W_{i, R_i'}$.	Therefore, 
	we define $\De_{i, R_i'}'(E, \cP_n)$ to be the number of pairs of $(q, q + 1)$-simplices $(\s,\wt \s)$ in $\cK_T(\cP_n)$ that are both contained in $W_{i, R_i'}$ and satisfy $(r(\s), r(\wt \s))\in E$  with $r(\wt \s) > r(\s)$. 

	We now continue from \eqref{dnz_eq} as follows.
	Since $\cP_n$ and $\cP_{i, n}$ share the same distribution, it suffices to derive bounds for $\De_{i,R_i'}'(E, \cP_n)$. Now, 
	$$\De_{i, R_i'}'(E, \cP_n) \le \cP(W_{i, R_i'}')^{2d + 2} \one\{\De_{i, R_i'}'(E, \cP_n) \ge 1\},$$
so that by H\"older's inequality,
	$$\EE[|\De_{i, R_i'}'(E, \cP_n)|^k] \le \big(\EE[(\cP (W_{i, R_i'}))^{16(2d + 2)k}]\big)^{1/16} \P(\De_{i,R_i'}'(E, \cP_n) \ge 1)^{15/16}.$$
By stationarity and the exponential decay of the radii of stabilization, the first quantity on the right-hand side is bounded by a finite constant not depending on $i$. Now, distinguishing on the value of $R_i'$,
\begin{align*}
\P(\De_{i, R_i'}'(E, \cP_n) \ge 1) &= \sum_{k\ge 1}\P\big(\{\De_{i, R_i'}'(E, \cP_n) \ge 1\}\cap\{k - 1 \le R_i' < k \}\big) 
\le \sum_{k \ge 1}\P(k - 1 \le R_i')^{1/16} \P(\De_{i, k}'(E, \cP_n) \ge 1)^{15/16},
\end{align*}
	so that by the exponential decay of the tails of $R_i'$, it remains to derive bounds on $\P(\De_{i, k}'(E, \cP_n) \ge 1)$. 
{The rest of the proof is separated in two parts, one treating the case for the directed filtrations, one for the \v C- and the VR-filtration.}

\noindent\textit{Directed filtrations}. 
In the case of the directed filtration, 
	$
\P(\De_{i, k }'(E, \cP_n) \ge 1) \le\int_{W_{i, k}}p_n^z(E_d)\d z,
$
	where $p_n^z(\Ed)$ denotes the probability that a branch born at $z = x + y$ has a lifetime in $\Ed$.  Moreover, it is bounded above by the probability that there exists at least one Poisson point in the interval $(x + \Ed)\times A$. Hence, $p_n^z(\Ed) \le |\Ed||A|$ and plugging this back into the integral concludes the first part of the proof.

It remains to consider the case where $i \not \in 2n\Eb$. Since $|\Eb| \ge n^{-3/4}$, this means that the distance between $i$ and the interval $\Eb$ is at least $n^{1/4}$. In particular, $
\De_{i,n}'(E) \ne 0$ only if $R_i' + T+ 1 \ge n^{1/4}$, where the external radius of stabilization $R_i'$ is as in \eqref{ext_stab_def}. Hence, by Cauchy-Schwarz,
\begin{align*}
		\EE[|D_{i, n}|^k] \le \E{\#\cP_n^k \one\{R_i' + T + 1\ge n^{1/4}\}} 
		\le \sqrt{\E{\#\cP_n^{2k}}}\sqrt{\P\big(R_0' + T + 1 > n^{1/4}\big)},
		\end{align*}
which is of order at most $O(n^{-16})$ due to the exponentially decaying tail of the radius of stabilization.

\noindent	\textit{\v C- and VR-filtration}.
	We now explain how to bound $\P(\De_{i, k }'(E, \cP_n) \ge 1)$ in the \v C- and VR-filtration.   To that end, we decompose this quantity further as 
	$$\De_{i, k}' (E, \cP_n) = \De_{i, k, 1}' (E, \cP_n) + \De_{i, k, 2}' (E, \cP_n),$$
	where the first summand takes into account pairs of simplices $(\s, \wt \s)$ sharing at least one common vertex, whereas in the second one, the simplices $\s$ and $\wt\s$ are disjoint. The arguments for bounding the contributions in $\De_{i, k, 2}'(E, \cP_n)$ are very similar to the ones presented in the proof of Proposition \ref{grid_prop}. Therefore, we now concentrate on the novel arguments required to deal with $\De_{i, k, 1}'(E, \cP_n)$. 

	First, note that there exists $0 \le q' \le q$ such that for the pair of simplices $(\s, \wt \s)$ from the definition of $\De_{i, k}'(E, \cP_n)$ we have $\s = \s_{q'} = \{Z_0, \dots,Z_{q'}, Y_{q' + 1}, \dots, Y_q\}$ and $\wt \s   = \{Z_0, \dots, Z_{q + 1}\}$ with $Z_0, \dots, Z_{q + 1},  Y_{q' + 1}, \dots, Y_q \subset \cP_n \cap W_{i, k}$ pairwise distinct. Next, we introduce the event
	$$A(z_0, \dots, z_{q + 1},  y_{q' + 1}, \dots,  y_{q }) :=\{(r(\s_{q'}), r(\wt \s)) \in E\}\cap \{ r(\wt \s) > r(\s_{q'})\}$$
and apply the Slivnyak-Mecke formula with
	$$f(z_0, \dots, z_{q + 1}, y_{q' + 1}, \dots, y_{q}) = 1\{A(z_0, \dots, z_{q + 1},  y_{q' + 1}, \dots, y_{q + 1})\} 1\big\{z_0, \dots, z_{q + 1}, y_{q' + 1}, \dots, y_{q } \in M_{i + j}\big\}.$$
Thus, arguing as in the proof of Proposition \ref{grid_prop}, 
\begin{align*}
	\EE[ \De_{i, k, 1}'(E, \cP_n) ]	&\le  \sum_{|j|\le k/2} \sum_{0 \le q' \le q}\EE \big[\# \{ (\s_{q'}, \wt \s):\, (r(\s_{q'}), r(\wt \s))\in E,\, r(\wt \s) > r(\s_{q'}),\,  \s_{q'} \cap ([i + j, i + j + 1]\times A) \ne \es \} \big] \\
			&\le \sum_{|j|\le k/2} \sum_{0 \le q' \le q}|M_{i + j}|^{2q - q' + 2}\P\big(A(Z_0',\dots,Z_{q + 1}', Y_{q' + 1}', \dots, Y_{q }' )\big),
			 \end{align*}
			 where the
	$Z_0',\dots,Z_{q + 1}', Y_{q' + 1}', \dots, Y_{q }'$ are now iid uniform in $M_{i + j}$. Thus, applying Proposition \ref{Cech_prop} shows that $ \EE[\De_{i, k, 1}'(E, \cP_n)]$ is bounded above by $ck|E|^{3/4}$ for a suitable $c > 0$. 
 \end{proof}

 %
 %
 The second ingredient in bounding the right-hand side of \eqref{bnDec} concerns derivations of covariance bounds, which is the topic of the following auxiliary result. It is motivated by the intuition that if the radii of stabilization exhibit exponential tails, then it is plausible that this decay is inherited by the martingale-difference terms corresponding to spatially distant locations.

 As in Lemma \ref{diffBoundLem}, the key to the covariance bounds lies in the exponential stabilization rather than in the cylindrical setup. For a similar statement under weaker conditions, we refer the reader to \cite[Theorem 1.11]{byy}.

\begin{lemma}[Covariance bound]
	\label{covBoundLem}
	For every $p, q \ge 1$ there exist $C_{p, q} C_{p, q}' > 0$ with the following property. Let $n \ge 1$ and $1 \le i_1 \le \cdots \le i_p < i_{p + 1} \le \cdots \le i_{p + q} $ and set 	 $X_1 = \prod_{k \le p}\D_{i_k, n}(E)$ and $X_2 = \prod_{k \le q}\D_{i_{p + k}, n}(E)$. Then, 
	$$\Cov\big(X_1, X_2\big) \le C_{p, q} \exp\big(-(i_{p + 1} - i_p)^{C_{p, q}'}\big)\sqrt{\E{X_1^4}\E{X_2^4}}.$$
\end{lemma}
\begin{proof}
	To derive covariance bounds, we represent the random variables $\D_{i, n}(E)$ through auxiliary Poisson point processes. More precisely, let $\{\cP^k\}_{k \ge 1}$ be a family of independent copies of $\cP$ and set
	\begin{align*}
		\cP_k^* &:= \big(\cP \cap (-\infty, i_k]\big) \cup \big(\cP^k \cap [i_k, {+}\infty)\big), \\
		\cP_k^{**} &:= \big(\cP \cap (-\infty, i_k - 1]\big) \cup \big(\cP^k \cap [i_k - 1, {+}\infty)\big)\\
		\D_{i_k, n}^*(E) &:= \b_n\big(E, \cP_k^*\big) - \b_n\big(E, \cP_k^{**}\big).
	\end{align*}
	Then, by construction, $\ms{Cov}[X_1,X_2] = \ms{Cov}[X_1^*,X_2^*]$,
	where 
	\begin{align*}
		X_1^* := \prod_{k \le p}\D_{i_k, n}^*(E)\qquad \text{ and }\qquad
		X_2^* := \prod_{k \le q}\D_{i_{p + k}, n}^*(E).
	\end{align*}
	Now, we let $R_k^*$ and $R_k^{**}$ denote the radii of stabilization associated with the shifted processes $\th_{i_k}(\cP^*_k)$ and $\th_{i_k}(\cP^{**}_k)$, respectively and put $R_k^\vee := R_k^* \vee R_k^{**}$. Furthermore, let
	$$E := \Big\{\bigcup_{k \le q}W_{R_{i_{p + k}}^\vee}(i_k) \subset [(i_p + i_{p + 1})/2, \infty)\Big\}$$
	denote the event that the influence zones for the second contribution extend at most to $(i_p + i_{p + 1})/2$ to the left.

	Then, we decompose the covariance as 
	$
		\Cov(X^*_1, X^*_2) = \Cov\big(X^*_1, X^*_2\one\{E\}\big) 		+ \Cov\big(X^*_1, X^*_2\one\{E^c\}\big) 
		$
	and observe that by definition of stabilization, the random variables $X^*_1$ and $X^*_2\one\{E\}$ are independent, so that the first covariance vanishes. It remains to bound the second summand.

	To that end, applying the Cauchy-Schwarz inequality gives that 
	$\Cov\big(X^*_1, X^*_2\one\{E^c\}\big) \le \sqrt{\Var\big(X^*_1\big)}\sqrt{\Var\big(X^*_2\one\{E^c\}\big)}.$
	A second application of Cauchy-Schwarz yields that 
	$\Var\big(X^*_2\one\{E^c\}\big)\le \sqrt{\E{(X^*_2)^4}}\sqrt{\P(E^c)}.$
	Hence, noting that the random variables $\{R_k^\vee\}_{k \ge 1}$ have exponential tails concludes the proof.
 \end{proof}

Equipped with Lemmas \ref{diffBoundLem} and \ref{covBoundLem}, we now derive the central variance and cumulant bounds. 
We recall from \cite[Identity (3.9)]{raic} that the mixed cumulant of random variables $Y_1, \dots, Y_4$ with finite fourth moment equals
\begin{align}
\label{E:C4M1.1}
	c^4(Y_1, \dots, Y_4) 
	 = \sum_{\{L_1, \dots, L_p\}\prec\{1, \dots, 4\}}(-1)^{p - 1}(p - 1)! \ \EE\Big[\prod_{i \in L_1}Y_i\Big] \cdots \EE\Big[\prod_{i \in L_p}Y_i \Big], 
\end{align}
where the sum runs over all partitions $L_1 \cup \cdots \cup L_p$ of $\{1, \dots, 4\}$.

\begin{proposition}[Variance and cumulant bound -- large blocks]
	\label{varProp}
	There exists a constant $C_{\ms{CV}} > 0$ with the following property. Let $n \ge 1$ and $E = \Eb \times \Ed \subset J$ be a block with $|\Eb| \wedge |\Ed| \ge n^{-3/4}$. Then, 
	$$\Var\big(\bar\b_n(E)\big) \vee  c^4\big(\bar\b_n(E)\big)\le C_{\ms{CV}}n |E|^{5/8}.$$
\end{proposition}

\begin{proof}[Proof of Proposition \ref{varProp}]
We prove the variance bound first. To ease notation, we put $C_{\ms M} = \max_{k \le 4} C_k$. Since $\{D_{i, n}(E)\}_i$ is a martingale-difference sequence, 
$$\Var\big(\bar\b_n(E)\big) = \sum_{|i| \le \kn } \EE[\D_{i, n}(E)^2].$$
For the directed filtration, part 1 of Lemma \ref{diffBoundLem} shows that the right-hand side is at most
$C_{\ms M}(n^{-15} +2 n |E|^{7/8}) \le 3C_{\ms M}n|E|^{5/8},$
thereby proving the asserted variance bound in the directed setting. For the \v C- and VR-filtration, we conclude from part 2 of Lemma \ref{diffBoundLem} that the right-hand side is at most the asserted $C_{\ms M}n|E|^{5/8}$.

Second, we prove the cumulant bound. To ease notation, we write $D_i$ for $D_{i, n}(E)$. The multilinearity of cumulants yields
\begin{align}\label{Cum_LargeBlocks}
	c^4(\bar\b_n(E)) \le \sum_{ i \le j \le k \le \ell }a_{i, j, k, \ell} \ c^4\big(\D_{i}, \D_{j}, \D_k, \D_\ell\big),
\end{align}
where the $a_{i, j, k, \ell} \ge 1$ are suitable combinatorial coefficients, depending only on which of the indices $i, j, k, \ell$ are equal.

In the following, we need to distinguish two cases in bounding the right-hand side of \eqref{Cum_LargeBlocks}: (a) $\ell - i \ge |E|^{-\de}$ and (b) $\ell - i < |E|^{-\de}$, where $\de = 1/64$. We begin with (a). First, we control the sum
$$\sum_{\substack{1 \le i \le j \le k \le \ell \\ \ell - i \ge |E|^{-\de}}}a_{i, j, k, \ell} \ c^4(\D_i, \D_j, \D_k, \D_\ell).$$
	In particular, $\max\{j - i, k - j, \ell - k\} \ge |E|^{-\de}/3$, and without loss of generality, we may focus on the case $j - i \ge |E|^{-\de} / 3$. To that end, we recall the semi-cluster decomposition from \cite{baryshnikov2005gaussian} or from \cite{raic}. Since we apply this decomposition not in the context of point processes, but just in the setting of sequences of random variables, it simplifies substantially. More precisely, $c^4(\D_i, \D_j, \D_k, \D_\ell)$ decomposes as
\begin{align*}
	c^4(\D_i, \D_j, \D_k, \D_\ell) 	=\hspace{-.5cm}\sum_{\{L_1, \dots, L_p\}\prec\{j, k, \ell\}} \hspace{-.5cm}a_{\{L_1, \dots, L_p\}}'\Cov\Big(\D_i, \prod_{s \in L_1}\D_s\Big)\EE\Big[\prod_{s \in L_2}\D_s\Big]\cdots \EE\Big[\prod_{s \in L_p}\D_s\Big]
		\end{align*}
for some coefficients $a_{\{L_1, \dots, L_p\}}'$ only depending on the structure of the partition, but not on the precise values of $j, k, \ell$. Hence, combining the moment bounds from Lemma \ref{diffBoundLem} and the covariance bounds from Lemma \ref{covBoundLem} concludes the case where $\ell - i \ge |E|^{-\de}$.

We continue with (b). It remains to bound the partial sum
	$$\sum_{\substack{i \le j \le k \le \ell \\ \ell - i \le |E|^{-\de}}}a_{i, j, k, \ell} \ c^4(\D_i, \D_j, \D_k, \D_\ell).$$
	consisting of those contributions where $\ell - i \le |E|^{-\de}$. To that end, leveraging the H\"older inequality, the representation in \eqref{E:C4M1.1} implies for a single cumulant the bound
\begin{align}
\label{sing_cum_bound}
        \big|c^4(\D_i, \D_j, \D_k, \D_\ell)\big| &\le \sum_{\{L_1, \dots, L_p\}\prec\{i, \dots, \ell\}}a_{\{L_1, \dots, L_p\}}'
         \prod_{h \in L_1} \EE[ |D_h|^{|L_1|} ]^{1/|L_1|} \ \cdots \prod_{h \in L_p} \EE[ |D_h|^{|L_p|} ]^{1/|L_p|}, 
 \end{align}
  where the coefficients $a_{\{L_1, \dots, L_p\}}'$ only depend on the structure of the partition but not on the precise values of $i,j, k, \ell$. Starting from this observation, we now argue a bit differently for the directed and for the \v C- resp.\ VR-filtration. 
  
First, we consider the directed filtration and deal with the contributions in \eqref{sing_cum_bound} with $i \in 2n\Eb$. Then, \eqref{sing_cum_bound} is at most 
  \begin{align}
  \label{mmax_bound_eq}
c \sup_{n \ge 1}\max_{h \le n}\max_{k \le 4}\EE[|D_h|^k]
  \end{align}
  for some $c > 0$ depending only on the chosen network model. Now, by part 1 of Lemma \ref{diffBoundLem}, this expression is at most 
  $c'|\Ed|^{7/8},$
  where $c' > 0$ again only depends on the choice of the model. Hence,
  $$\sum_{\substack{ i \le j \le k \le \ell \le n \\ \ell - i \le |E|^{-\de}\\ i \in 2n\Eb}}a_{i, j, k, \ell} \ c^4(\D_i, \D_j, \D_k, \D_\ell)\le 2c'n|\Eb||\Ed|^{7/8}|E|^{-3\de},$$
  and the right-hand side is in $O(|E|^{5/8}).$
  
  Second, consider the contributions in \eqref{sing_cum_bound} with $i \not \in 2n\Eb$. If, for instance, $L_1$ is the element of the partition containing $i$, then, by part 1 of Lemma \ref{diffBoundLem}, 
  $\EE[ |D_i|^{|L_1|} ]^{1/|L_1|}\le C_{\ms M}^{1/4}n^{-4}.$
  Since the other moments also remain bounded, we conclude that 
  $$\sum_{\substack{ i \le j \le k \le \ell \le n \\ \ell - i \le |E|^{-\de}\\ i \not \in 2n\Eb}}a_{i, j, k, \ell}c^4(\D_i, \D_j, \D_k, \D_\ell)\in O(n^{-3}|E|^{-3\de}).$$
  In particular, since $|E| \ge n^{-3/2}$, the above expression is in $O(n|E|^{5/8})$. 

	Finally, the \v C- resp.\ VR setting is a bit simpler, as we do not need to distinguish between different cases. Indeed, arguing as in the first case, we arrive at the bound \eqref{mmax_bound_eq} for any $i$ with $|i| \le \kn$. By part 2 of Lemma \ref{diffBoundLem}, the latter is at most $c''|E|^{11/16}|E|^{-3\de}$ for some $c'' > 0$. Hence, aggregating over all $i$ with $|i| \le \kn$ yields the desired upper bound of order $O(n|E|^{5/8})$.
 \end{proof}

\begin{proof}[Proof of Theorem \ref{T:GaussianLimitPoisson1} and Theorem \ref{T:GaussianLimitPoisson2}]
	First, Proposition \ref{grid_prop} reduced the verification of the central moment bound \eqref{lav_eq} to blocks $E := \Eb \times \Ed \subset J$ satisfying $|\Eb| \wedge |\Ed| \ge n^{-3/4}$. Now, applying the cumulant identity 
 $
 	\EE[X^4] = 3\Var(X)^2 + c^4(X)
$ for $X := n^{-1/2}\bar\b_n(E)$ and inserting the bounds from Proposition \ref{varProp}, we arrive at 
\begin{align*}
	n^{-2}\EE\big[\bar\b_n(E)^4\big] &\le 3C_{\ms{CV}}^2 |E|^{5/4} + C_{\ms{CV}}n^{-1}|E|^{5/8} 
	 \le 3C_{\ms{CV}}^2 |E|^{5/4} + C_{\ms{CV}} |E|^{5/4+1/24} \le C |E|^{5/4}.
\end{align*}
Moreover, asymptotic normality of the finite-dimensional distributions is discussed in Appendix \ref{AppendixA}.
	Applying Theorem \ref{bw_thm} (which is \cite[Theorem 3]{bickel1971convergence}) thus concludes the proof.
 \end{proof}

\section*{Acknowledgments}
We thank two anonymous referees for reading through the entire manuscript very thoroughly and for providing us with highly valuable and constructive feedback. Their suggestions and comments improved the quality of the presentation substantially.
In particular, one referee provided us with a motivation to extend the functional CLT to the PD of the \v C-complex.


\appendix
\section{Multivariate asymptotic normality}\label{AppendixA}
To verify the multivariate asymptotic normality, we follow the strategy from Theorem 3.1 in \cite{penrose2001central}, which itself relies on the classical martingale-CLT from \cite{mcleish1974dependent}. More precisely, \cite{penrose2001central} first discretize the space into boxes, then decompose the functional into martingale differences, which are obtained from this discretization and finally apply the martingale CLT. We sketch only the most important steps, referring the reader to \cite{penrose2001central} for details.

First, by the Cram\'er-Wold device, it suffices to establish the CLT for linear combinations of the form
\begin{align*}
	    n^{-1/2}\sum_{j \le k} \a_j \ \big( \bn^{r_j, s_j} - \EE[\bn^{r_j, s_j}]\big),
\end{align*}
with $(r_1, s_1), \dots, (r_k, s_k) \in J$ and $\a_1, \dots, \a_k \in \R$.

To apply the martingale CLT, we decompose the centered Betti numbers as
$$ \sum_{j\le k} \a_j \big( \bn^{r_j, s_j}- \EE\big[\bn^{r_j, s_j} \big]\big) = \sum_{|i|\le n/2} D_{i, n}, $$
where the martingale differences are defined in complete analogy to \eqref{dnDef} now taking into account the linear combination with the $\alpha_j$. We assume that $n$ is odd in order to avoid heavy notation.

As mentioned above, the key tool in the proof of multivariate asymptotic normality is the martingale CLT of Theorem 2.3 from \cite{mcleish1974dependent}. We restate here the version of Theorem 2.10 from \cite{penrose}, specialized to mean-zero martingales. 
\begin{theorem}
	\label{mcl_thm}
	Suppose that for each $n \ge 1$ the sequence $M_{1, n}, \dots, M_{n, n}$ is a mean-zero martingale with respect to some filtration and set $X_{i, n} := M_{i, n} - M_{i - 1, n}$ with $M_{0, n} = 0$. Suppose that 
\begin{itemize}
	\item [(a)] $\sup_{n\ge 1} \EE[ \max_i X_{i, n} ^2 ] < \infty $,
\item [(b)] $  \max_i |X_{i, n}| \to 0 $ in probability as $n\to\infty$,
\item [(c)] $ \sum_{i } X_{i, n}^2 \to \wt \s^2 $ in $L^1(\p)$ for some $\wt\s^2 > 0$.
\end{itemize}
	Then, $M_{n, n} \to \cN(0, \wt\s^2)$ in distribution.
\end{theorem}
Henceforth, we verify conditions (a)--(c) for $X_{i, n} := n^{-1/2}D_{i - (n + 1)/2, n}$.
To establish conditions (a) and (b) in the martingale CLT, we first need uniform moment bounds.
%
%
\begin{lemma}[Uniformly bounded moments]\label{L:UniformBounded1}
	It holds that
\begin{align}\label{E:UniformBounded1}
	\sup_{n \ge 0} \sup_{|i| \le \kn}\EE\big[\D_{i, n}^4\big] < \infty.
 \end{align}
\end{lemma}
Loosely speaking, Lemma \ref{L:UniformBounded1} is a consequence of stabilization. In fact, in Lemma \ref{diffBoundLem}, we have derived a more refined upper bound. Nevertheless, to make the presentation self-contained, we include a proof for the case of the directed filtration. The arguments for the \v C- and VR-filtrations are analogous. Put $\dnz^{r, s} := \bn^{r, s} - \bnz^{r, s}$.
\begin{proof}
	By definition of $\D_{i, n}^4$, it suffices to show that 
	$\sup_{n \ge 0} \sup_{|i| \le \kn}\EE\big[(\dnz^{r,s})^4\big] < \infty.$
	Relying on the external radius of stabilization, we obtain as in Lemma \ref{diffBoundLem} that
	$$\EE[(\dnz^{r,s})^4] \le \EE\big[\big((\De_{i,R_i'}^{r, s})'(\cP_n) + (\De_{i,R_i'}^{r, s})'(\cP_{i, n})\big)^4\big] \le 32\EE\big[(\De_{i,R_i'}^{r, s})'(\cP_n)^4\big],$$
	where $(\De_{i, R_i'}^{r, s})'(\cP_n)$ denotes the number of branches in the network on $\cP_n$ with life time at least $s$ and born  in the domain $W_{i, R_i'}'$. Now, we apply the H\"older inequality as in Lemma \ref{diffBoundLem} to arrive at 
$$\EE[(\De_{i, R_i'}^{r, s})'(\cP_n)^4] \le \big(\EE[(\cP (W_{i, R_i'}))^{64}]\big)^{1/16}.$$
By stationarity, the right-hand side does not depend on $i$, and it is finite as the external radii of stabilization have exponential tails.
\end{proof}

%
%
For condition (c), we need a form of weak stabilization of persistent Betti numbers. In the setting of the persistent Betti numbers in general dimensions such a result dates back to \cite{hiraoka2018limit}. In the current cylindrical set-up, the proof is far simpler due to the absence of percolation phenomena.
\begin{lemma}[Weak stabilization]\label{L:WeakStabilization1}
Let $i\in\Z$. Then, $\dnz^{r, s}$ converges a.s.~to an a.s.~finite random variable $  \De_{i, \infty}^{r, s}$.
Moreover, the convergence is uniform in the sense that for every $\e > 0$ and $(r, s) \in J$,
$$
	\lim_{n \to \infty}\sup_{i\in \Z: |i|\le \kn - \e n} \P\big( \dnz^{r, s} \ne \De_{i, \infty}^{r, s}  \big) =0.
$$
\end{lemma}
\begin{proof}
	In the case of a directed filtration, we have a radius of stabilization by assumption. In the case of the \v C- and VR-filtration, all the work for this result has already been done in Section \ref{proof_sec} through the construction of the radius of stabilization defined in \eqref{cech_stab}. In particular, if $n$ is so large that $\th_{-i}(W_{R_i}) \subset W_n$, then $\dnz$ remains unchanged and the uniform convergence follows from stationarity.
 \end{proof}

\begin{proof}[Verifying the conditions of Theorem \ref{mcl_thm}]
$\phantom{a}$\\
	\textit{Condition (a) and (b):}
Lemma \ref{L:UniformBounded1} yields the boundedness of
\begin{align*}
	&\sup_{n\ge 1} \ n^{-1} \ \EE\Big[ \max_{|i| \le \kn}  D_{i, n} ^2 \Big] \le \sup_{n\ge 1} \ n^{-1} \ \sum_{|i|\le n/2} \EE\big[ D_{i, n}^2 \big] = O(1) \quad\text{ and } \quad n^{-2} \ \EE\Big[ \max_{|i| \le \kn}  D_{i, n}^4 \Big] \le n^{-2} \sum_{|i|\le n/2} \EE\big[ D_{i, n}^4 \big] = O(n^{-1}).
\end{align*}
\textit{Condition (c):} For the convergence requirement, it is enough to study the convergence of the expression
\begin{align}\label{E:MVN0}
 &n^{-1} \sum_{|i| \le \kn} \EE\big[ \bn^{r,s} - \bnz^{r,s} \ | \ \cG_{i} \big] \EE\big[ \bn^{u, v} - \bnz^{u, v} \ | \ \cG_{i}\big] 
\end{align}
for two possibly distinct pairs $(r, s), (u, v)$ in $J$. Using the weak stabilizing property from Lemma~\ref{L:WeakStabilization1} and the pointwise ergodic theorem, one finds with similar ideas as in the proof of Penrose and Yukich that the sum tends to
\begin{align}
  \gamma( (r, s), (u, v) ) := \E{ \EE\big[ \De_{0,\infty}^{r, s} | \cG_0 \big] \EE\big[ \De_{0, \infty}^{u, v} | \cG_0 \big]} \quad \text{a.s.~and in } \cL^1(\p). \label{Gamma_Def}
\end{align}
We sketch here the steps and refer for a detailed guidance to \cite{krebs2019asymptotic} in Proposition 5.5, Lemma 5.6 and Lemma 5.7. 
Now, the expression \eqref{E:MVN0} equals
\begin{align}
\label{E:MVN1}
  n^{-1} \sum_{|i|\le n/2} \Big\{ \EE\big[ \dnz^{r, s} \ | \ \cG_{i} \big] \ \EE\big[ \dnz^{u, v} \ | \ \cG_{i} \big] - \EE \big[\De_{i ,\infty}^{ r, s} \ | \ \cG_{i} \big] \EE\big[ \De_{i, \infty}^{ u, v} \ | \ \cG_{i} \big] \Big\}
	+ n^{-1} \sum_{|i|\le n/2} \EE[ \De_{i, \infty}^{r, s} \ | \ \cG_i ] \EE[ \De_{i ,\infty}^{u, v} \ | \ \cG_i ]. 
\end{align}
First, observe that the second sum converges to a $ \gamma( (r, s), (u, v) ) $ a.s.~and in $\cL^1(\p)$ by the ergodic theorem as the sequence $\big( \EE[ \De_{i,\infty}^{r,s} \ | \ \cG_i ] : i \in \Z\big)$ is stationary and ergodic for $r\le s$.

Moreover, H\"older's and Jensen's inequality, give for each summand in \eqref{E:MVN1} the following upper bound 
\begin{align*}
 \big\| \EE\big[ \dnz^{r,s} \ | \ \cG_{i} \big] \EE\big[ \dnz^{u, v} \ | \ \cG_{i} \big] - 
  \EE\big[ \De_{i,\infty}^{r, s} \ | \ \cG_{i} \big] \EE\big[ \De_{i,\infty}^{u, v} \ | \ \cG_{i} \big] \big\|_{\cL^1(\p)} 
	&\le \EE\big[ (\dnz^{r, s} - \De_{i,\infty}^{r, s} )^2 \big]^{1/2} \EE\big[ (\dnz^{u, v}) ^2\big]^{1/2} \\
	&\phantom\le+ \EE\big[ (\dnz^{u, v} - \De_{i,\infty}^{u, v} )^2 \big]^{1/2} \EE\big[ (\De_{i,\infty}^{r, s}) ^2 \big]^{1/2},
\end{align*}
where the right-hand side tends to 0 as $n \to \infty$ because of the uniform bounded moments condition from Lemma~\ref{L:UniformBounded1} and the weak stabilizing property from Lemma~\ref{L:WeakStabilization1}. Moreover, this convergence is even uniform in the sense made precise in Lemma \ref{L:WeakStabilization1}. Hence, \eqref{E:MVN1} converges in the $\cL^1$-norm to $\gamma( (r,s), (u,v))$, so the limit in Condition (c) is $\wt \s^2 = \sum_{i, j\le k} \a_i \a_j \ \gamma( (r_i, s_i), (r_j, s_j))$. 
 \end{proof}
\end{document}